\documentclass{amsart}
\usepackage{amsthm}
\usepackage{amsmath}
\usepackage{amssymb}
\usepackage{comment}
\usepackage{enumitem}
\usepackage{amsfonts}
\usepackage{mathtools}
\usepackage{hyperref}
\usepackage{todonotes}
\usepackage[utf8]{inputenc}

\DeclareMathOperator{\conv}{Conv}
\DeclareMathOperator{\wid}{width}
\DeclareMathOperator{\cay}{Cay}

\DeclareMathOperator{\di}{\delta^{\text{ideal}}}
\DeclareMathOperator{\bi}{\textit{b}^{\text{ideal}}}

\newcommand{\mc}{\mathcal}

\newcommand{\bs}{\boldsymbol}

\newcommand{\GL}{\mathrm{GL}}
\newcommand{\height}{\mathrm{ht}}
\newcommand{\R}{\mathbb{R}}
\newcommand{\Vol}{\mathrm{Vol}}
\newcommand{\id}{\mathrm{id}}
\renewcommand{\O}{\mathcal{O}}

\newcommand\precdot{\mathrel{\ooalign{$\prec$\cr
  \hidewidth\raise0.001ex\hbox{$\cdot\mkern0.6mu$}\cr}}}

\newtheorem{theorem}{Theorem}[section]
\newtheorem{def-prop}[theorem]{Definition-Proposition}
\newtheorem{prop}[theorem]{Proposition}
\newtheorem{conj}[theorem]{Conjecture}
\newtheorem{lemma}[theorem]{Lemma}
\newtheorem{cor}[theorem]{Corollary}

\theoremstyle{definition}
\newtheorem{ex}[theorem]{Example}
\newtheorem{quest}[theorem]{Question}
\newtheorem{defin}[theorem]{Definition}

\theoremstyle{remark}
\newtheorem*{remark}{Remark}

\title{Balance constants for Coxeter groups}

\author{Christian Gaetz}
\thanks{C.G. was supported by a National Science Foundation Graduate Research Fellowship under Grant No. 1122374.}
\address{Department of Mathematics, Cornell University, Ithaca, NY, USA.}
\email{\href{mailto:crgaetz@gmail.com}{{\tt crgaetz@gmail.com}}}
\author{Yibo Gao}
\address{Beijing International Center for Mathematical Research, Peking University, Beijing, CN.}
\email{\href{mailto:gaoyibo@bicmr.pku.edu.cn}{{\tt gaoyibo@bicmr.pku.edu.cn}}}

\date{\today}

\begin{document}

\begin{abstract}
The $1/3$-$2/3$ Conjecture, originally formulated in 1968, is one of the best-known open problems in the theory of posets, stating that the \emph{balance constant} of any non-total order is at least $1/3$.  By reinterpreting balance constants of posets in terms of convex subsets of the symmetric group, we extend the study of balance constants to convex subsets $C$ of any Coxeter group.  Remarkably, we conjecture that the lower bound of $1/3$ still applies in any finite Weyl group, with new and interesting equality cases appearing.

We generalize several of the main results towards the $1/3$-$2/3$ Conjecture to this new setting: we prove our conjecture when $C$ is a weak order interval below a fully commutative element in any acyclic Coxeter group (a generalization of the case of width-two posets), we give a uniform lower bound for balance constants in all finite Weyl groups using a new generalization of order polytopes to this context, and we introduce \emph{generalized semiorders} for which we resolve the conjecture.

We hope this new perspective may shed light on the proper level of generality in which to consider the $1/3$-$2/3$ Conjecture, and therefore on which methods are likely to be successful in resolving it.
\end{abstract}
\maketitle
\section{Introduction}

\subsection{The $1/3$-$2/3$ Conjecture}

Given a finite poset $P$ on $n$ elements, a \emph{linear extension} of $P$ is an order-preserving bijection $\lambda:P \to [n]$, where $[n]=\{1,2,\ldots,n\}$ has the natural order on the integers.  Let $x,y \in P$ and consider the quantity
\begin{equation}
\label{eq:def-delta-for-posets}
\delta_P(x,y)=\frac{|\{\text{linear extensions } \lambda:P \to [n] \text{ such that } \lambda(x) > \lambda(y) \}|}{|\{\text{linear extensions } \lambda:P \to [n]\}|}.
\end{equation}

The quantity $\delta$ is of considerable interest.  If $P$ represents random partial information about some underlying total order on the same ground set, then $\delta_P(x,y)$ gives information about the probability that $x$ actually precedes $y$ in the total order.  Unfortunately, neither the numerator nor denominator may be easily computed, as computing the number of linear extensions of a poset is known \cite{Brightwell-Winkler} to be $NP$-hard, even \cite{Dittmer-Pak} for the important case of two-dimensional posets which we will encounter later.  

Despite these difficulties, one could hope that there exist $x,y \in P$ with 
\[
\min(\delta_P(x,y), 1-\delta_P(x,y))
\]
large, so that the additional information of whether $x$ precedes $y$ would reduce the number of linear extensions, and thus the remaining uncertainty, as much as possible.  It thus makes sense to study the \emph{balance constant}
\[
b(P)=\max_{x,y} \min(\delta_P(x,y), 1-\delta_P(x,y)),
\]
to establish ``information-theoretic" bounds for the above problem \cite{Linial}.  It is in this context that the following well known conjecture has been made three times independently, by Kislicyn in 1968, Freedman in 1974, and Linial in 1984 \cite{Kislicyn,Linial}.

\begin{conj}[The $1/3$-$2/3$ Conjecture]
\label{conj:1-3-2-3}
For any finite poset $P$ which is not a total order, $b(P) \geq \frac{1}{3}$.
\end{conj}

This conjecture has received considerable attention, with many weaker bounds and special cases having been established; see Brightwell's survey \cite{Brightwell} for a discussion.  

\subsection{Posets as convex subsets of the symmetric group}
\label{sec:intro-convex-sets}

Let $P$ be a poset on $\{p_1,\ldots, p_n\}$, so any linear extension $\lambda: P \to [n]$ may be thought of as a permutation $w_{\lambda} \in S_n$ with $w_{\lambda}(i)=\lambda(p_i)$.  Let
\[
C(P)=\{w_{\lambda} \: | \: \lambda \text{ a linear extension of $P$}\}
\]
be the set of these permutations as $\lambda$ ranges over all linear extensions of $P$; clearly the set $C(P)$ determines the poset $P$.  It is a folklore fact that the sets $C \subseteq S_n$ which come in this way from a poset are exactly the \emph{convex} subsets: those $C$ such that any element on a minimum-length path in the Cayley graph $\cay(S_n,S)$ between two elements of $C$ in the also lies in $C$. Here $\cay(S_n,S)$ is the Cayley graph for $S_n$ generated by the set $S$ of simple transpositions $(i \: {i+1})$.  In the dual picture of the braid arrangement, these are exactly the $C$ such that the union of the closed regions corresponding to the permutations in $C$ is a (convex) cone.

Taking the perspective of posets as convex subsets of the symmetric group, notice that 
\[
\delta_P(p_i,p_j)=\frac{|\{w \in C(P) \: | \: w(i) > w(j)\}|}{|C|}.
\]
The right-hand-side is the fraction of permutations in $C(P)$ with a given \emph{inversion}, and this perspective admits a natural generalization to convex sets in any Coxeter group.

\subsection{Balance constants for Coxeter groups}
See Section~\ref{sec:background} for background and definitions on Coxeter groups and Weyl groups.

For any convex set $C$ in a Coxeter group $W$ and any reflection $t$ we define
\[
\delta_C(t)=\frac{|C_t|}{|C|},
\]
where $C_t=\{w \in C \: | \: \ell(wt)<\ell(w)\}$ is the set of elements in $C$ having $t$ as an inversion; when $C$ is clear we often simply write $\delta(t)$.  In light of the discussion in Section~\ref{sec:intro-convex-sets} this definition is exactly analogous to (\ref{eq:def-delta-for-posets}), recovering $\delta_P(x,y)$ when $W=S_n$, $C$ is the convex set associated to $P$, and $t$ is the transposition swapping $\lambda(x)$ and $\lambda(y)$.  It thus makes sense to define the \emph{balance constant}
\[
b(C)=\max_t \min(\delta_C(t), 1-\delta_C(t)).
\]
Remarkably, when $W$ is a finite Weyl group, Conjecture~\ref{conj:1-3-2-3} appears still to hold:

\begin{conj}\label{conj:1-3-2-3-weyl}
Let $W$ be any finite Weyl group and $C \subseteq W$ a convex set which is not a singleton. Then $b(C) \geq \frac{1}{3}$.
\end{conj}

\begin{remark}
Conjecture~\ref{conj:1-3-2-3-weyl} could perhaps be extended to all finite Coxeter groups (equivalently, all finite real reflection groups).  However, the only new irreducible cases this would include are the dihedral groups, for which the conjecture is obvious, and the exceptional group of type $H_4$ (and its parabolic subgroup of type $H_3$), for which the conjecture becomes a finite check. The group $W$ of type $H_4$ has 60 reflections; the convex sets containing $\id$ are the $W^A$ (see Section~\ref{sec:background-cox-groups}) for $A$ a subset of the reflections. Thus a naive verification of the conjecture would involve iterating over these $2^{60}$ subsets, which is computationally intractable. We have been unable to cut down this computation sufficiently in order to complete it, although it may be within reach of a cleverer computational approach.

We often choose to work in the setting of finite Weyl groups to avoid additional technicalities; the results of Section~\ref{sec:fully-commutative}, though, apply in all finite Coxeter groups, and more generally in Coxeter groups with acyclic Coxeter diagrams. 
\end{remark}

In this paper we show that many known partial results towards Conjecture~\ref{conj:1-3-2-3} can be generalized to the context of Conjecture~\ref{conj:1-3-2-3-weyl}, or even further.  

The remainder of the paper is organized as follows: in Section~\ref{sec:background} we give background on Coxeter groups, Weyl groups, and convex subsets of these.  Section~\ref{sec:fully-commutative} resolves Conjecture~\ref{conj:1-3-2-3-weyl} in the case $C$ is an interval below a \emph{fully commutative} element of $W$ in the increased generality of acyclic Coxeter groups; this generalizes a classical result of Linial \cite{Linial} that Conjecture~\ref{conj:1-3-2-3} holds for width-two posets.  Section~\ref{sec:equality-cases} also gives examples of convex sets achieving equality in Conjecture~\ref{conj:1-3-2-3-weyl}; this is a richer set of examples than exists for posets, where there is (conjecturally) only one irreducible example.  Section~\ref{sec:semiorders} proves Conjecture~\ref{conj:1-3-2-3-weyl} in the context of \emph{generalized semiorders}, which we introduce; this is a generalization of Brightwell's result \cite{Brightwell-semiorder} for semiorder posets.  Section~\ref{sec:bounds-Weyl} gives a type-independent proof of a uniform lower bound $b(C) \geq \varepsilon > 0$ for Conjecture~\ref{conj:1-3-2-3-weyl}; this is inspired by Kahn and Linial's proof for posets \cite{Kahn-Linial} and relies on a new generalization of order polytopes.  Our generalized semiorders and generalized order polytopes may be of independent interest. Finally, Section~\ref{sec:examples} gives examples which shed light on whether the above results might be expected in even greater generality.

An extended abstract describing part of this work appears in the proceedings of FPSAC 2021 \cite{fpsac-abstract}.

\section{Background on Coxeter groups and root systems}
\label{sec:background}

We will work in several different levels of generality throughout the paper.  This section collects background and notation for general Coxeter systems in Section \ref{sec:background-cox-groups} and for finite Weyl groups (finite crystallographic Coxeter groups) and their associated root systems in Section \ref{sec:background-weyl-groups}.

\subsection{Coxeter groups}
\label{sec:background-cox-groups}
We follow the conventions of \cite{Bjorner-Brenti}. Let $(W,S)$ be a Coxeter system.  We write $\Gamma$ for the associated \emph{Coxeter diagram}, the graph with vertex set $S$ and an edge labelled $m_{ij}$ between vertices $s_i$ and $s_j$ whenever the quantity $m_{ij}$ giving the defining relation $(s_is_j)^{m_{ij}}=\id$ of $W$ is at least $3$.  We say $W$ is \emph{acyclic} if the graph $\Gamma$ contains no cycles, and \emph{irreducible} if $\Gamma$ is connected.  

For $w \in W$, the \emph{length} $\ell(w)$ is the smallest number $\ell$ such that $w=s_1 \cdots s_{\ell}$ with the $s_i \in S$.  Such an expression of minimal length is called a \emph{reduced word} or \emph{reduced expression}.  The \emph{left (resp. right) weak order} is the partial order on $W$ with cover relations $u \lessdot_L su$ (resp. $u \lessdot_R us)$ whenever $\ell(su)=\ell(u)+1$ (resp. $\ell(us)=\ell(u)+1$) and $s \in S$.  We write $\cay_L(W)$ and $\cay_R(W)$ for the left and right Cayley graphs for $W$ with respect to the generating set $S$, viewing these as undirected graphs, and often identifying them with the Hasse diagrams of the weak orders.

The set $T$ of conjugates of elements of $S$ is called the set of \emph{reflections}.  For $w \in W$ the \emph{right (resp. left) inversion set} $T_R(w)$ is $\{t \in T \: | \: \ell(wt)<\ell(w)\}$ (resp. $\{t \in T \: | \: \ell(tw)<\ell(w)\}$).  It is well known that
\[
|T_R(w)|=|T_L(w)|=\ell(w),
\]
and that weak order is characterized by containment of inversions sets:
\begin{align*}
   u &\leq_L v \iff T_R(u) \subseteq T_R(v) \\
   u &\leq_R v \iff T_L(u) \subseteq T_L(v).
\end{align*}

For $D \subseteq A \subseteq T$ we write $W_D^A$ for the set of elements in $W$ whose right inversion set lies between $D$ and $A$:
\[
W_D^A=\{w \in W \: | \: D \subseteq T_R(w) \subseteq A\}.
\]
The reader should not confuse this notation with similar notation, which is often used in other sources, for parabolic subgroups and quotients of $W$.

A subset $C \subseteq W$ is \emph{left (resp. right) convex} if it is convex with respect to the metric on $W$ determined by the natural graph distance in $\cay_L(W)$ (resp. $\cay_R(W)$).  That is, $C$ is convex if all elements of $W$ which lie on some minimum-length path in $\cay_L$ between $u$ and $v$ lie in $C$ whenever $u$ and $v$ do.  The \emph{left (resp. right) convex hull} $\conv_L(w_1,\ldots, w_d)$ (resp. $\conv_R(w_1,\ldots, w_d)$) of a collection of elements $w_1, \ldots, w_d \in W$ is defined to be the intersection of all left (resp. right) convex subsets of $W$ which contain $\{w_1,\ldots,w_d\}$; the convex hull is itself clearly convex.

\begin{remark}
We make the convention that when ``left" and ``right" are not specified it is assumed that we are working with left weak order, left Cayley graphs and convex sets, and so on.
\end{remark}

The following characterization of convex subsets of $W$ is due to Tits \cite{Tits}; another proof, with language and notation closer to our own, is given in \cite{Bjorner-Wachs-gen-quo}.

\begin{theorem}[Tits \cite{Tits}]
\label{thm:tits-convexity}
A set $C \subseteq W$ is left convex if and only if it is of the form $W_D^A$ for some $D \subseteq A \subseteq T$.
\end{theorem}

If $W=W_1 \times W_2$ is a reducible Coxeter group, convex sets $C \subseteq W$ are products $C_1 \times C_2$ of convex sets $C_1 \subseteq W_1$ and $C_2 \subseteq W_2$.  This implies that $b(C)=\min(b(C_1), b(C_2))$, so it suffices to consider $W$ irreducible in Conjecture~\ref{conj:1-3-2-3-weyl}.

As the action of $W$ on $\cay_L$ by right multiplication is by graph automorphisms, it is clear that $C$ is convex if and only if $C \cdot w=\{cw \: | \: c \in C\}$ is for every $w \in W$.  Thus, choosing any $c \in C$ we may consider the translated convex set $C \cdot c^{-1}$ which now contains the identity and is equivalent to $C$ for the purposes of Conjecture~\ref{conj:1-3-2-3-weyl}.  Convex sets containing $\id$ are exactly the convex order ideals in left weak order, and by Theorem~\ref{thm:tits-convexity} these are the sets $W_{\emptyset}^A$, for which we will often write simply $W^A$. 

\subsection{Finite Weyl groups and crystallographic root systems}
\label{sec:background-weyl-groups}

For some of our results we will need to take advantage of additional structure present for finite crystallographic Coxeter groups (finite Weyl groups), some of which is outlined here.  We refer readers to Humphreys \cite{Humphreys} for a detailed exposition on the classical theory of root systems and Weyl groups.

Let $\Phi\subset E$ be a finite crystallographic root system of rank $r$, where $E$ is an ambient Euclidean space of dimension $r$ with a positive definite symmetric bilinear form $\langle-,-\rangle$, with a chosen set of positive roots $\Phi^+\subset\Phi$ and the corresponding simple roots $\Delta=\{\alpha_1,\ldots,\alpha_r\}\subset\Phi^+$. Let the \emph{fundamental coweights} $\omega_1^{\vee},\ldots,\omega_r^{\vee}$ be a dual basis of $\Delta$ with respect to $\langle-,-\rangle$, i.e., $\langle\alpha_i,\omega_j^{\vee}\rangle=\delta_{ij}$ where $\delta_{ij}$ is the Kronecker delta. 
For each root $\alpha\in\Phi$, the \emph{reflection across} $\alpha$, which can be thought of as reflection across the hyperplane orthogonal to $\alpha$, is defined as 
$$s_{\alpha}:x\mapsto x-\frac{2\langle \alpha,x\rangle}{\langle\alpha,\alpha\rangle}\alpha\in\GL(E).$$
The Weyl group $W=W(\Phi)$ is a finite subgroup of $\GL(E)$ generated by the $s_{\alpha}$, for $\alpha\in\Phi$, or equivalently, generated by the simple reflections $S=\{s_1,\ldots,s_r\}$ where $s_i=s_{\alpha_i}$.  

The root system $\Phi$ is called \emph{irreducible} if it cannot be partitioned into two proper subsets $\Phi_1\sqcup\Phi_2$ such that $\langle \beta_1,\beta_2\rangle=0$ for all $\beta_1\in\Phi_1$ and $\beta_2\in\Phi_2.$ Irreducible root systems are completely classified, with four infinite families, types $A_n$ ($n\geq1$), $B_n$ ($n\geq2$), $C_n$ ($n\geq2$) and $D_n$ ($n\geq4$), and five exceptional types, $E_6$, $E_7$, $E_8$, $F_4$ and $G_2$. Let $e_i$ denote the $i$-th standard basis vector in $\R^n$. We adopt the following conventions for the infinite families of irreducible root systems:
\begin{itemize}
\item Type $A_{n-1}$: $\Phi=\{e_i-e_j\:|\:1\leq i\neq j\leq n\}$, $\Phi^+=\{e_i-e_j\:|\:1\leq i<j\leq n\}$, $\Delta=\{e_i-e_{i+1}\:|\:1\leq i\leq n-1\}$.
\item Type $B_n$: $\Phi=\{\pm e_i\pm e_j,\ \pm e_i\:|\:1\leq i\neq j\leq n\}$, $\Phi^+=\{e_i\pm e_j,\ e_i\:|\:1\leq i<j\leq n\}$, $\Delta=\{e_i-e_{i+1}\:|\:1\leq i\leq n-1\}\cup\{e_n\}$.
\item Type $C_n$: $\Phi=\{\pm e_i\pm e_j,\ \pm 2e_i\:|\:1\leq i\neq j\leq n\}$, $\Phi^+=\{e_i\pm e_j,\ 2e_i\:|\:1\leq i<j\leq n\}$, $\Delta=\{e_i-e_{i+1}\:|\:1\leq i\leq n-1\}\cup\{2e_n\}$.
\item Type $D_n$: $\Phi=\{\pm e_i\pm e_j\:|\:1\leq i\neq j\leq n\}$, $\Phi^+=\{e_i\pm e_j\:|\:1\leq i<j\leq n\}$, $\Delta=\{e_i-e_{i+1}\:|\:1\leq i\leq n-1\}\cup\{e_{n-1}+e_n\}$.
\end{itemize}

The pair $(W,S)$ forms a finite Coxeter system, so the material of Section~\ref{sec:background-cox-groups} can be applied. Although the root systems of type $B_n$ and $C_n$ are different, the corresponding Coxeter systems are isomorphic. 

\begin{defin}
The \emph{root poset} is the partial order on the positive roots $\Phi^+$ such that $\alpha\leq\beta$ if $\beta-\alpha$ is a non-negative linear combination of the simple roots $\Delta$.  We will abuse notation by simply writing $\Phi^+$ for the root poset $(\Phi^+, \leq)$.
\end{defin}

The minimal elements of $\Phi^+$ are the simple roots $\Delta$. It is a classical fact that there exists a unique maximum $\xi$ in the root poset, called the \emph{highest root}. For $\alpha\in\Phi$, its \emph{height} is
\[
\height(\alpha)\coloneqq\langle\alpha,\omega_1^{\vee}+\cdots+\omega_r^{\vee}\rangle.
\]
That is, if $\alpha=c_1\alpha_1+\cdots+c_r\alpha_r$ then $\height(\alpha)=c_1+\cdots+c_r$. Clearly positive roots have positive heights and negative roots have negative heights, and the crystallographic condition ensures that all heights are integers, giving the root poset $\Phi^+$ a grading by height. The height of the root system $\Phi$ is the height of its highest root $\xi$, and we write $\height(\Phi)=\height(\xi)$. 

For $w\in W(\Phi)$, its \emph{inversion set} is $I_{\Phi}(w)=\{\alpha\in\Phi^+\:|\: w\alpha\in\Phi^-\}$; this is in bijection with $T_R(w)$ via the map $\alpha \leftrightarrow s_{\alpha}$ and we therefore sometimes write $C_{\alpha}$ for $C_{s_{\alpha}}$ and $\delta_C(\alpha)$ for $\delta_C(s_{\alpha})$.  For $A \subseteq \Phi^+$ we also write $W^A$ for the convex set $W^{\{s_a \: | \: a \in A\}}$ defined in Section~\ref{sec:background-cox-groups}. The \emph{Coxeter arrangement} is the central hyperplane arrangement in $E$ consisting of the hyperplanes
\[
H_{\alpha}\coloneqq\{x\in E\:|\:\langle x,\alpha\rangle=0\}
\]
for each $\alpha\in\Phi^+$. Regions of the Coxeter arrangement are called \emph{Weyl chambers} and are in bijection with the Weyl group $W$ in the following manner: for $w\in W$, the corresponding open Weyl chamber is
\[
R_w\coloneqq\{x\in E\:|\: \langle x,\alpha\rangle>0\text{ for }\alpha\in\Phi^+\setminus I_{\Phi}(w),\langle x,\alpha\rangle<0\text{ for }\alpha\in I_{\Phi}(w)\}.
\]
In particular, the \emph{fundamental Weyl chamber} $R_{\id}$ is $\{x\in E\:|\:\langle x,\alpha\rangle>0,\alpha\in\Phi^+\}$. It is now clear that the left Cayley graph $\cay_L(W)$ on $W$ is the same as the graph on the Weyl chambers where $R_w$ is connected to $R_u$ if and only if they are separated by a single hyperplane $H_{\alpha}$. As a result, $C\subseteq W$ is left convex if and only if $\bigcup_{w\in C}\overline{R_w}$ is a convex polyhedron. This geometric perspective will be used in Section~\ref{sec:bounds-Weyl}.

\section{The fully commutative case}
\label{sec:fully-commutative}

\subsection{Width-two posets and $321$-avoiding permutations}
\label{sec:width-two}
An \emph{antichain} in a poset $P$ is a collection of pairwise incomparable elements; the size of the largest antichain is the \emph{width} $\wid(P)$ of $P$.  The following result of Linial establishes Conjecture~\ref{conj:1-3-2-3} in the case $\wid(P)=2$.

\begin{theorem}[Linial \cite{Linial}]
\label{thm:width-2}
Let $P$ have width two. Then $b(P) \geq \frac{1}{3}$.
\end{theorem}

Although the width-two condition is very restrictive, all known equality cases $b(P)=\frac{1}{3}$ for Conjecture~\ref{conj:1-3-2-3} lie within this class of posets.  Indeed, it is conjectured \cite{Kahn-Saks} that for each $k>2$ there is a lower bound for $b(P)$ on width-$k$ posets which is strictly greater than $\frac{1}{3}$, with these bounds approaching $\frac{1}{2}$ as $k \to \infty$, so that Theorem~\ref{thm:width-2} covers those posets which are (conjecturally) closest to violating Conjecture~\ref{conj:1-3-2-3} (see the survey by Brightwell \cite{Brightwell} for a heuristic discussion).

The \emph{dimension} $\dim(P)$ of $P$ is the smallest number $d$ such that 
\[
C(P)=\conv(w_1,\ldots,w_d)
\]
for some $w_1,\ldots,w_d \in S_n$, or equivalently the smallest number of linear extensions needed to uniquely determine the poset $P$.  It was shown by Dilworth \cite{Dilworth} that any finite poset $P$ has dimension at most its width:
\[
\dim(P) \leq \wid(P).
\]
In particular, any poset of width two has order dimension two (the only posets of dimension one are the total orders, and these have width one).  

A poset $P=\{p_1,\ldots,p_n\}$ is \emph{naturally labelled} if the map $p_i \mapsto i, \forall i$ is a linear extension. Any naturally labelled two-dimensional poset $P$ has 
\[
C(P)=\conv(\id,w)=[\id,w]_L
\]
for some $w \in S_n$, and it is immediate from the definition of $C(P)$ that $P$ has width two if and only if the permutation $w$ avoids the pattern $321$, meaning that there are no $1 \leq i < j < k \leq n$ such that $w(i)>w(j)>w(k)$.

In this section we will generalize Theorem~\ref{thm:width-2} to all Coxeter groups $W$ with acyclic Coxeter diagrams; the role of $321$-avoiding permutations will be played by \emph{fully commutative elements} of $W$, introduced by Stembridge \cite{Stembridge}.

\subsection{Fully commutative elements in Coxeter groups}

For $(W,S)$ any Coxeter system and $w \in W$, we write $\mc{R}_w$ for the set of reduced words of $w$.  A well known result of Tits \cite{Tits-words} implies that all elements of $\mc{R}_w$ are connected by relations of the form
\[
s_is_j \cdots = s_js_i \cdots
\]
with $m_{ij} \geq 2$ factors on each side.  Applying such a relation to a reduced word is called a \emph{commutation move} when $m_{ij}=2$ and a \emph{braid move} otherwise.  Allowing only commutation moves determines an equivalence relation $\sim$ on $\mathcal{R}_w$, and the elements of $\mathcal{R}_w / \sim$ are called \emph{commutation classes}.

\begin{defin}[Stembridge \cite{Stembridge}]
An element $w \in W$ is called \emph{fully commutative} if $\mc{R}_w$ consists of a single commutation class.  Equivalently, $w$ is fully commutative if no reduced word for $w$ admits a braid move. 
\end{defin}

\begin{prop}[Stembridge \cite{Stembridge}]
\label{prop:fc-iff-321-avoider}
For $W=S_n$, a permutation $w$ is fully commutative if and only if it avoids the pattern $321$.
\end{prop}

Theorem~\ref{thm:fully-comm} is the main theorem of this section, establishing Conjecture \ref{conj:1-3-2-3-weyl} for intervals below fully commutative elements in acyclic Coxeter groups; it is proven in Section~\ref{sec:fc-theorem}.  By Proposition~\ref{prop:fc-iff-321-avoider} and the discussion in Section~\ref{sec:width-two}, Theorem~\ref{thm:fully-comm} generalizes Theorem~\ref{thm:width-2}, which is the case $W=S_n$.

\begin{theorem}
\label{thm:fully-comm}
Let $W$ be a (not necessarily finite) Coxeter group with acyclic Coxeter diagram, and let $w \in W$ be a non-identity fully commutative element.  For the convex set $C=[\id,w]_L$ we have
\[
b(C)\geq \frac{1}{3}.
\]
\end{theorem}

\subsection{Heaps and weak order intervals}
\label{sec:heaps}

Let $(W,S)$ be any Coxeter system. Given an element $w \in W$ with a reduced word $w=s_{i_1} \cdots s_{i_{\ell}}=\bs{s} \in \mc{R}_w$ the associated \emph{heap poset} (or just \emph{heap}) $H_{\bs{s}}$ is the partial order $([\ell], \preceq)$ which is the transitive closure of the following relation:
\[
j \preceq k \text{ if $j \geq k$ and $m_{i_j i_k} \neq 2$},
\]
where the $m_*$'s are the edge labels of the Coxeter diagram of $W$.  We also give heaps the structure of $S$-labelled posets by associating with the vertex $j$ the simple reflection $s_{i_j}$.  As we may have $i_j=i_{j'}$, multiple vertices may receive the same $S$-label.  Heaps were introduced in a more general setting by Viennot \cite{Viennot} and applied to reduced words in Coxeter groups by Stembridge \cite{Stembridge}.

\begin{prop}[Stembridge \cite{Stembridge}]
\label{prop:heap-depend-on-commutation-class}
The heap poset $H_{\bs{s}}$, as an $S$-labelled poset, depends only on the commutation class of $\bs{s}$.  
\end{prop}

By Proposition~\ref{prop:heap-depend-on-commutation-class}, if $w$ is fully commutative the heap $H_{\bs{s}}$ does not depend on the chosen reduced word $\bs{s} \in \mc{R}_w$.  Thus we may refer unambiguously to the heap of $w$, which we denote by $H_w$.

Suppose that $w$ is fully commutative and $u \in [\id,w]_L$; clearly $u$ is also fully commutative.  Choose any reduced word $\bs{s'}=s_{i_{\ell'+1}} \cdots s_{i_{\ell}}$ for $u$ and construct the heap $H_w$ with respect to an augmented reduced word
\[
\bs{s}=s_{i_1} \cdots s_{i_{\ell'}} s_{i_{\ell'+1}} \cdots s_{i_{\ell}}=w.
\]
Since $\bs{s'}$ is a suffix of $\bs{s}$, it is immediate from the definitions that $H_w$ contains the $S$-labelled poset $H_u$ as an order ideal.  Thus we have constructed a map
\begin{align*}
\Psi: [\id,w]_L &\to J(H_w) \\
u & \mapsto H_u
\end{align*}
where, for a finite poset $P$, we write $J(P)$ for the distributive lattice of order ideals of $P$. (See \cite[\S 3.4]{ec1} for basic properties of $J(P)$.)  Furthermore, it is clear that $\Psi$ is order preserving.

\begin{prop}[Stembridge \cite{Stembridge}]
\label{prop:weak-order-interval-from-heap}
Let $w \in W$ be fully commutative. Then
\[
\Psi: [\id,w]_L \to J(H_w)
\]
is an isomorphism of posets.
\end{prop}

We may identify the poset $[\id,w]_L$ with the set of inversion sets $T_R(u)$ for $u \in [\id,w]_L$ ordered by containment. The isomorphism $\Psi$ between distributive lattices induces a bijection
\begin{align}
\label{eq:Psi-and-psi}
\psi: T_R(w) &\to H_w  \\
s_{i_{\ell}}s_{i_{\ell-1}}\cdots s_{i_{a+1}} s_{i_a} s_{i_{a+1}} &\cdots s_{i_{\ell-1}} s_{i_{\ell}} \mapsto a \nonumber
\end{align}
such that for any $u \in [\id,w]_L$ we have
\[
\Psi(u)=\{\psi(t) \: | \: t \in T_R(u)\} \in J(H_w).
\]

This observation immediately implies the following proposition, which will be crucial to our proof of Conjecture~\ref{conj:1-3-2-3-weyl} for fully commutative elements in acyclic Coxeter groups in Section \ref{sec:fc-theorem}.

\begin{prop}
\label{prop:inversion-fraction-equals-ideal-fraction}
Let $W$ be any Coxeter group. Let $w\in W$ be fully commutative and $C=[\id,w]_L$; let $\psi: T_R(w) \to H_w$ be as in (\ref{eq:Psi-and-psi}).  Then for any $t \in T_R(w)$ and $x=\psi(t) \in H_w$ we have
\[
\delta_C(t)=\frac{|\{I \in J(H_w) \: | \: x \in I\}|}{|J(H_w)|}.
\]

\end{prop}

In light of Proposition~\ref{prop:inversion-fraction-equals-ideal-fraction} we will use the notation
\begin{align*}
    \di(x; Q) &= \frac{|\{I \in J(Q) \: | \: x \in I\}|}{|J(Q)|}, \\
    \bi(Q) &= \max_x \min(\di(x; Q), 1-\di(x; Q)),
\end{align*}
and sometimes write just $\di(x)$ when $Q$ is clear from context.  Thus we have reduced the proof of Theorem~\ref{sec:fc-theorem} to showing that 
\begin{equation}
\label{eq:b-ideal}
    \bi(H_w) \geq \frac{1}{3}
\end{equation}
whenever $w$ is a non-identity fully commutative element in an acyclic Coxeter group $W$.

\subsection{Balance constants for fully commutative intervals}
\label{sec:fc-theorem}
For a finite poset, such as a heap poset, whose elements are labelled by elements of $S$, we write $\sigma(x)$ for the label of the element $x$.

The following lemma is implicit in \cite{Stembridge} and is the reason that heap posets as we have defined them fit into Viennot's more general theory \cite{Viennot}.

\begin{lemma}
\label{lem:heaps-and-graph}
Let $(W,S)$ be any Coxeter system with Coxeter diagram $\Gamma$ and $w \in W$ any element.  Let $\bs{s} \in \mc{R}_w$ be a reduced word and $H=H_{\bs{s}}$ the heap poset with elements labelled by $S$.  Then
\begin{itemize}
    \item[(a)] If $x \precdot x'$ is a cover relation in $H$, then $\sigma(x)$ and $\sigma(x')$ are adjacent in $\Gamma$.
    \item[(b)] If $s, s'$ are adjacent in $\Gamma$ or $s=s'$, and $x,x' \in H$ have labels $s,s'$ respectively, then $x,x'$ are comparable in $H$.
\end{itemize}
\end{lemma}

The exact bound $b \geq \frac{1}{3}$ becomes relevant in the next lemma.

\begin{lemma}
\label{lem:common-up-degree}
Let $Q$ be a finite poset with $\bi(Q) < \frac{1}{3}$.  If $x$ is maximal among those elements of $Q$ such that $\di(x)>\frac{2}{3}$, then $x$ is covered by at least two elements of $Q$.  Dually, if $y$ is minimal among those elements of $Q$ such that $\di(y)<\frac{1}{3}$, then $y$ covers at least two elements of $Q$. 
\end{lemma}
\begin{proof}
Suppose, for the sake of contradiction, that $y$ covers at most one element of $Q$.  If $y$ covers a single element $z$, then we have an injection:
\begin{align*}
    \mathcal{I}_1 \coloneqq \{I \in J(Q) \: | \: z \in I, y \not \in I\} &\hookrightarrow \{I \in J(Q) \: | \: y \in I\} \eqqcolon \mathcal{I}_2 \\
    I &\mapsto I \cup \{y\}.
\end{align*}
Note that $\di(z)> \frac{2}{3}$ by the minimality of $y$ and our assumption on $\bi(Q)$. Thus
\[
|\mathcal{I}_1| = \left(\di(z) - \di(y)\right) |J(Q)| > \frac{1}{3} |J(Q)|.
\]
On the other hand, we have
\[
|\mathcal{I}_2| = \di(y) |J(Q)| < \frac{1}{3}|J(Q)|,
\]
a contradiction.

If $y$ does not cover any elements, a similar contradiction is obtained by removing the condition that $z \in I$ in $\mathcal{I}_1$.  The claim for $x$ follows from that for $y$ by replacing $Q$ with the dual poset $Q^*$; this operation preserves $\bi$ while swapping the roles of $x$ and $y$.
\end{proof}

We are now ready to complete the proof of Theorem~\ref{thm:fully-comm}.

\begin{proof}[Proof of Theorem~\ref{sec:fc-theorem}]
Suppose $W$ is an irreducible acyclic Coxeter group and $w \in W$ is a non-identity fully commutative element such that (\ref{eq:b-ideal}) fails:
\[
\bi(H_w) < \frac{1}{3}.
\]
Let $\Gamma$ denote the Coxeter diagram of $W$, a connected acyclic graph, that is, a tree.  Call an element $z \in H_w$ \emph{common} if $\di(z) > \frac{2}{3}$ and \emph{uncommon} if $\di(z) < \frac{1}{3}$; by our assumption on $\bi(H_w)$, all elements are either common or uncommon.  Let $X$ denote the set of maximal common elements and $Y$ the set of minimal uncommon elements.

Define a directed graph $G$ on $X \sqcup Y$ as follows: for each $x \in X$, let $z,z'$ be two uncommon elements of $H_w$ covering $x$ guaranteed to exist by Lemma~\ref{lem:common-up-degree}.  Choose elements $y, y' \in Y$ such that $y \preceq z$ and $y' \preceq z'$; these exist since any uncommon element is above some element of $Y$ by definition.  Add the directed edges $x \to y$ and $x \to y'$ to $G$.  This determines all directed edges from $X$ to $Y$.  Dually, for each $y \in Y$, let $z,z'$ be two common elements of $H_w$ covered by $y$ guaranteed to exist by Lemma~\ref{lem:common-up-degree}.  Choose elements $x, x' \in X$ such that $z \preceq x$ and $z' \preceq x'$.  Add the directed edges $y \to x$ and $y \to x'$ to $G$.  Thus every vertex in $G$ has out degree two.

We now observe a key property of the graph $G$ and the $S$-labels of its elements. For $x \in X$, let $x \to y$ and $x \to y'$ be edges in $G$ and $z,z'$ the intermediate elements covering $x$ discussed above. We claim that $\sigma(y)$ and $\sigma(y')$ lie in different connected components of $\Gamma \setminus \{\sigma(x)\}$.  To see this, observe first that $\sigma(z) \neq \sigma(z')$ by Lemma~\ref{lem:heaps-and-graph} (b), since $z,z'$ are incomparable in $H_w$.  As $\sigma(z), \sigma(z')$ both neighbor $\sigma(x)$ in $\Gamma$ by Lemma~\ref{lem:heaps-and-graph} (a) and as $\Gamma$ is a tree, we see that $\sigma(z)$ and $\sigma(z')$ lie in different connected components of $\Gamma \setminus \{\sigma(x)\}$.  The property then follows by noting that $y$ and $z$ lie in the same connected component of $\Gamma \setminus \{\sigma(x)\}$, as do $y'$ and $z'$.  This is true because Lemma~\ref{lem:heaps-and-graph} (a) implies that a sequence of cover relations $y \precdot \cdots \precdot z$ in $H_w$ gives a path from $\sigma(y)$ and $\sigma(z)$ in $\Gamma$; this path may not pass through $\sigma(x)$ or part (b) of the lemma would contradict the fact that $z$ covers $x$.  Of course $G$ also has the dual property: for $y \in Y$, the labels of the elements $x,x'$ with $y \to x$ and $y \to x'$ in $G$ lie in different connected components of $\Gamma \setminus \{\sigma(y)\}$.

Now, choose an element $s$ in the tree $\Gamma$ to designate as a root, and define the \emph{depth} of $s' \in \Gamma$ as its distance from the root.  Let $g$ be any element of $G$ and let $g \to h, g \to h'$ be its two outgoing edges.  By the argument above, $\sigma(h)$ and $\sigma(h')$ lie in different connected components of $\Gamma \setminus \{\sigma(g)\}$.  All nodes of depth less than $\sigma(g)$ belong to the component of $\Gamma \setminus \{\sigma(g)\}$ that contains $s$.  Thus either $\sigma(h)$ or $\sigma(h')$ must have depth strictly greater than that of $\sigma(g)$, say $\sigma(h)$.  But we may repeat this process, beginning now with $h$ rather than $g$, to obtain an infinite sequence of elements of $\Gamma$ of increasing depth; since $\Gamma$ is a finite tree, this is a contradiction.
\end{proof}

\subsection{Equality cases}
\label{sec:equality-cases}
Let $P_3$ denote the (unique up to isomorphism) poset on three elements with a single cover relation.  It is clear that $P_3$ achieves equality in Conjecture~\ref{conj:1-3-2-3} (see Figure~\ref{fig:type-A-equality}).  In fact, it was shown by Aigner \cite{Aigner} that the only equality cases in Conjecture~\ref{conj:1-3-2-3} among width-two posets occur when $P$ is an ordinal sum of some number of copies of $P_3$ and some number of singleton posets, and it is generally believed (see Brightwell \cite{Brightwell}) that these are the only equality cases among all finite posets. Figures~\ref{fig:type-D-equality},\ref{fig:type-B-equality}, and \ref{fig:type-E-equality} below show that there is a much richer collection of equality cases in Theorem~\ref{thm:fully-comm}.

\begin{figure}[ht!]
\centering
\begin{tikzpicture}
\node[draw,shape=circle,fill=black,scale=0.5](a)[label=below: {$s_1$}] at (-0.3,0) {};
\node[draw,shape=circle,fill=black,scale=0.5](b)[label=below: {$s_2$}] at (1,0) {};

\draw (b)--(a) node [midway, fill=white] {$3$};
\end{tikzpicture}
\hspace{0.5in}
\begin{tikzpicture}
\node[draw,shape=circle,fill=black,scale=0.5](a)[label=right: {$s_2$}] at (0,0) {};
\node[draw,shape=circle,fill=black,scale=0.5](d)[label=right: {$s_1$}] at (0,1) {};

\draw (a)--(d);
\end{tikzpicture}
\caption{The Coxeter diagram (left) for the Weyl group of type $A_2$ (the symmetric group $S_3$). The fully commutative element $w=s_1s_2=231$ has heap poset $H_w$ shown on the right with $S$-labels; this example is an equality case $b([\id,w]_L)=\bi(H_w)=\frac{1}{3}$ corresponding to the poset $P_3$ in the original formulation of the $1/3$-$2/3$ Conjecture.}
\label{fig:type-A-equality}
\end{figure}
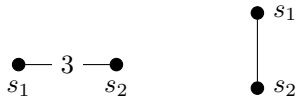

\begin{figure}[ht!]
\centering
\begin{tikzpicture}
\node[draw,shape=circle,fill=black,scale=0.5](a)[label=below: {$s_1$}] at (-0.3,0) {};
\node[draw,shape=circle,fill=black,scale=0.5](b)[label=below: {$s_2$}] at (1,0) {};
\node[draw,shape=circle,fill=black,scale=0.5](c)[label=above: {$s_3$}] at (2,0.5) {};
\node[draw,shape=circle,fill=black,scale=0.5](d)[label=below: {$s_4$}] at (2,-0.5) {};

\draw (b)--(a) node [midway, fill=white] {$3$};
\draw (c)--(b) node [midway, fill=white] {$3$};
\draw (b)--(d) node [midway, fill=white] {$3$};
\end{tikzpicture}
\hspace{0.5in}
\begin{tikzpicture}
\node[draw,shape=circle,fill=black,scale=0.5](a)[label=right: {$s_2$}] at (0,0) {};
\node[draw,shape=circle,fill=black,scale=0.5](b)[label=below: {$s_1$}] at (0.5,-1) {};
\node[draw,shape=circle,fill=black,scale=0.5](c)[label=below: {$s_3$}] at (-0.5,-1) {};
\node[draw,shape=circle,fill=black,scale=0.5](d)[label=right: {$s_4$}] at (0,1) {};

\draw (b)--(a);
\draw (c)--(a);
\draw (a)--(d);
\end{tikzpicture}
\caption{The Coxeter diagram (left) for the Weyl group of type $D_4$. The fully commutative element $w=s_4s_2s_3s_1$ has heap poset $H_w$ shown on the right with $S$-labels; this example is an equality case $b([\id,w]_L)=\bi(H_w)=\frac{1}{3}$.}
\label{fig:type-D-equality}
\end{figure}

\begin{figure}[ht!]
\centering
\begin{tikzpicture}
\node[draw,shape=circle,fill=black,scale=0.5](a)[label=below: {$s_1$}] at (0,0) {};
\node[draw,shape=circle,fill=black,scale=0.5](b)[label=below: {$s_2$}] at (1,0) {};
\node[draw,shape=circle,fill=black,scale=0.5](c)[label=below: {$s_3$}] at (2,0) {};

\draw (b)--(a) node [midway, fill=white] {$3$};
\draw (c)--(b) node [midway, fill=white] {$4$};
\end{tikzpicture}
\hspace{0.5in}
\begin{tikzpicture}
\node[draw,shape=circle,fill=black,scale=0.5](a)[label=right: {$s_2$}] at (0,0) {};
\node[draw,shape=circle,fill=black,scale=0.5](b)[label=below: {$s_1$}] at (0.5,-1) {};
\node[draw,shape=circle,fill=black,scale=0.5](c)[label=below: {$s_3$}] at (-0.5,-1) {};
\node[draw,shape=circle,fill=black,scale=0.5](d)[label=right: {$s_3$}] at (0,1) {};

\draw (b)--(a);
\draw (c)--(a);
\draw (a)--(d);
\end{tikzpicture}
\caption{The Coxeter diagram (left) for the Weyl group of type $B_3$. The fully commutative element $w=s_3s_2s_3s_1$ has heap poset $H_w$ shown on the right with $S$-labels; this example is an equality case $b([\id,w]_L)=\bi(H_w)=\frac{1}{3}$.}
\label{fig:type-B-equality}
\end{figure}

\begin{figure}[ht!]
\centering
\begin{tikzpicture}
\node[draw,shape=circle,fill=black,scale=0.5](a)[label=below: {$s_1$}] at (0,0) {};
\node[draw,shape=circle,fill=black,scale=0.5](b)[label=below: {$s_2$}] at (1,0) {};
\node[draw,shape=circle,fill=black,scale=0.5](c)[label=below: {$s_3$}] at (2,0) {};
\node[draw,shape=circle,fill=black,scale=0.5](d)[label=below: {$s_4$}] at (3,0) {};
\node[draw,shape=circle,fill=black,scale=0.5](e)[label=below: {$s_5$}] at (4,0) {};

\node[draw,shape=circle,fill=black,scale=0.5](f)[label=above: {$s_6$}] at (2,1) {};

\draw (b)--(a) node [midway, fill=white] {$3$};
\draw (b)--(c) node [midway, fill=white] {$3$};
\draw (c)--(d) node [midway, fill=white] {$3$};
\draw (d)--(e) node [midway, fill=white] {$3$};
\draw (c)--(f) node [midway, fill=white] {$3$};
\end{tikzpicture}
\hspace{0.5in}
\begin{tikzpicture}
\node[draw,shape=circle,fill=black,scale=0.5](a)[label=below: {$s_1$}] at (-1,0) {};
\node[draw,shape=circle,fill=black,scale=0.5](b)[label=below: {$s_3$}] at (0,0) {};
\node[draw,shape=circle,fill=black,scale=0.5](c)[label=below: {$s_5$}] at (1,0) {};

\node[draw,shape=circle,fill=black,scale=0.5](d)[label=left: {$s_2$}] at (-0.5,1) {};
\node[draw,shape=circle,fill=black,scale=0.5](e)[label=right: {$s_4$}] at (0.5,1) {};

\node[draw,shape=circle,fill=black,scale=0.5](f)[label=right: {$s_3$}] at (0,2) {};

\node[draw,shape=circle,fill=black,scale=0.5](g)[label=right: {$s_6$}] at (0,3) {};

\draw (a)--(d)--(b)--(e)--(c);
\draw (d)--(f)--(e);
\draw (f)--(g);
\end{tikzpicture}
\caption{The Coxeter diagram (left) for the Weyl group of type $E_6$. The fully commutative element $w=s_6s_3s_2s_4s_1s_3s_5$ has heap poset $H_w$ shown on the right with $S$-labels; this example is an equality case $b([\id,w]_L)=\bi(H_w)=\frac{1}{3}$.}
\label{fig:type-E-equality}
\end{figure}

One striking feature of Conjecture~\ref{conj:1-3-2-3-weyl} is that the conjectured bound $b(C) \geq \frac{1}{3}$ is type-independent.  Since large finite Weyl groups contain large type $A$ parabolic subgroups, one possible explanation for this would be if there were some larger type-dependent lower bound on $b$ for convex sets which are, say, ``genuinely type $B$", but the type $A$ parabolic subgroup would ensure that the overall bound could be no larger than $\frac{1}{3}$.  This is not the case.  Instead, Figures \ref{fig:type-D-equality}, \ref{fig:type-B-equality}, and \ref{fig:type-E-equality} show that all finite Weyl groups have fully commutative elements $w$ such that $b([\id,w]_L)=\frac{1}{3}$ which do not come from known type $A$ equality cases.  We think these examples make Conjecture~\ref{conj:1-3-2-3-weyl} all the more interesting: there are several genuinely distinct ways to match and yet not surpass the conjectured bound.

Comparing these examples with Aigner's results prompts the following question.

\begin{quest}
Let $w$ be a fully commutative element in a finite Weyl group with 
\[
b([\id,w]_L)=\bi(H_w)=\frac{1}{3}.
\]
Is $H_w$ isomorphic to a disjoint union of the heap posets appearing in Figures \ref{fig:type-A-equality}, \ref{fig:type-D-equality}, \ref{fig:type-B-equality}, and \ref{fig:type-E-equality}?
\end{quest}

\begin{ex}
\label{ex:infinite-equality-cases}
If we consider not just Weyl groups as in Figures~\ref{fig:type-A-equality}, \ref{fig:type-D-equality}, \ref{fig:type-B-equality}, and \ref{fig:type-E-equality}, but also infinite acyclic Coxeter groups (to which Theorem~\ref{thm:fully-comm} still applies), there are infinitely many non-isomorphic connected heap posets giving equality cases.  For any $k$ and $\ell$, let $H$ be the $S$-labelled poset shown below, and let $W$ be the acyclic Coxeter group whose Coxeter diagram is isomorphic to the Hasse diagram of $H$, with all edge labels equal to 3. Let $w=s_{k+{\ell}}s_{k+{\ell}-1} \cdots s_2 s_1$ be the fully commutative element with heap $H$. 

Let $y \in H$ be the element labelled by $s_{k+1}$ and let $x$ be any of the elements labelled $s_1,\ldots,s_k$. We have $|J(H)|=\ell+2^k$, $\di(x)=(\ell+2^{k-1})/(\ell+2^k)$, and $\di(y)=\ell/(\ell+2^k)$. Thus, setting $\ell=2^{k-1}$ yields $\di(x)=\frac{2}{3}$ and $\di(y)=\frac{1}{3}$, so $\bi(H)=\frac{1}{3}$.

\begin{center}
\begin{tikzpicture}
\node[draw,shape=circle,fill=black,scale=0.5](a1)[label=below: {$s_1$}] at (-1,-1) {};
\node[draw,shape=circle,fill=black,scale=0.5](a2)[label=below: {$s_2$}] at (-.7,-1) {};
\node [label=below: {$\cdots$}] at (0,-.75) {};
\node[draw,shape=circle,fill=black,scale=0.5](a3)[label=below: {$s_k$}] at (1,-1) {};

\node[draw,shape=circle,fill=black,scale=0.5](b1)[label=right: {$s_{k+1}$}] at (0,0) {};
\node[draw,shape=circle,fill=black,scale=0.5](b2)[label=right: {$s_{k+2}$}] at (0,.5) {};
\node(b3)[label={$\vdots$}] at (0,1) {};
\node[draw,shape=circle,fill=black,scale=0.5](b4)[label=right: {$s_{k+\ell-1}$}] at (0,2) {};
\node[draw,shape=circle,fill=black,scale=0.5](b5)[label=right: {$s_{k+\ell}$}] at (0,2.5) {};

\draw (a1)--(b1);
\draw (a2)--(b1);
\draw (a3)--(b1);
\draw (b1)--(b2)--(b3);
\draw (b4)--(b5);
\end{tikzpicture}
\end{center}
Although these examples give the smallest possible balance constants for fully commutative intervals in their respective acyclic Coxeter groups, balance constants of other convex sets may be lower, unlike the conjectured situation for Weyl groups; see the examples in Section \ref{sec:examples}.
\end{ex}

\section{Generalized semiorders}\label{sec:semiorders}

Let us recall the notion of a \emph{semiorder} (also known in the literature as a \emph{unit interval order}).

\begin{defin}\label{def:semiorder}
A finite poset $P$ is a \emph{semiorder} if there exists a function $f:P\rightarrow\R$ such that $x<y$ in $P$ if and only if $f(y)-f(x)\geq1$. 
\end{defin}

It is a standard fact that $P$ is a semiorder if and only if $P$ avoids induced copies of the posets $\mathbf{2}+\mathbf{2}$ and $\mathbf{3}+\mathbf{1}$. This characterization says that, informally, semiorders are ``tall and thin" so they are good candidates to serve as counterexamples to the $1/3$-$2/3$ Conjecture (see the discussion in \cite{Brightwell}). However semiorders have been ruled out as counterexamples.

\begin{theorem}[Brightwell~\cite{Brightwell-semiorder}]
\label{thm:brightwell-semiorder}
For any semiorder $P$ that is not a total order, $b(P)\geq\frac{1}{3}$. 
\end{theorem}

In this section, we generalize semiorders to arbitrary finite Weyl groups and show that their balance constants are at least $\frac{1}{3}$. Our notation is the same as in Section~\ref{sec:background-weyl-groups}.
\begin{defin}
Let $\Phi$ be a root system with Weyl group $W$. A convex set $C\subseteq W$ is a \emph{generalized semiorder} if $C=W^A$ for some order ideal $A\subseteq\Phi^+$ of the root poset.
\end{defin}

We first note that generalized semiorders include the classical semiorders (Definition~\ref{def:semiorder}). Indeed, given a poset $P$ on $n$ elements and a function $f:P\rightarrow\R$, let $p_1,\ldots,p_n$ be the elements of $P$, indexed such that 
\[
f(p_1)\leq f(p_2)\leq\cdots\leq f(p_n).
\]
Recall that the type $A_{n-1}$ root system has positive roots 
\[
\Phi^+=\{e_i-e_j\:|\:1\leq i<j\leq n\}.
\]
It is easy to see that $e_{i'}-e_{j'}\leq e_{i}-e_{j}$ in the root poset if and only if $i\leq i'<j'\leq j$. We have $C(P)=W^A$, where the set of allowed inversions is
\[
A=\{e_i - e_j \in \Phi^+ \: | \: f(p_j)-f(p_i)<1\}.
\]
This subset $A\subseteq\Phi^+$ must be an order ideal in $\Phi^+$: if $e_i-e_j\in A$ and $e_{i'}-e_{j'}\leq e_i-e_j$ in the root poset, then $f(p_{j'})\leq f(p_j)$ and $f(p_{i'})\geq f(p_i)$, so $f(p_{j'})-f(p_{i'})\leq f(p_j)-f(p_i)<1$, meaning $e_{i'}-e_{j'}\in A$.

The following generalization of Theorem~\ref{thm:brightwell-semiorder} is our main theorem of this section.

\begin{theorem}\label{thm:semiorder}
Let $C\subseteq W$ be a generalized semiorder with $|C|>1$. Then $b(C)\geq\frac{1}{3}$.
\end{theorem}

The rest of the section will be devoted to proving Theorem~\ref{thm:semiorder}. We start with the following simple but important lemma.

\begin{lemma}\label{lem:semiorder-uncommon}
Let $C\subseteq W(\Phi)$ be a generalized semiorder. Then $\delta_C(\alpha)\leq\frac{1}{2}$ for all $\alpha\in\Phi^+$.
\end{lemma}
\begin{proof}
Let $C=W^A$ where $A\subseteq\Phi^+$ is an order ideal. If $\alpha\notin A$, then $\delta_C(\alpha)=0$, so suppose $\alpha\in A$. It suffices to construct an injection from $C_{\alpha}\coloneqq\{w\in C\:|\: \alpha\in I_{\Phi}(w)\}$ to $C\setminus C_{\alpha}$; we claim that $w\mapsto ws_{\alpha}$ works. It is clear that if $w$ has $\alpha$ as an inversion, then $ws_{\alpha}$ does not have $\alpha$ as an inversion, so the task is to show that if $w\in C_{\alpha}$, then $ws_{\alpha}\in C$, which we do in the contrapositive.

Suppose that $ws_{\alpha}\notin C$. This means that there exists $\beta\in \Phi^+\setminus A$ such that $\beta\in I_{\Phi}(ws_{\alpha})$. In other words,
$$ws_{\alpha}\beta=w\left(\beta-\frac{2\langle\alpha,\beta\rangle}{\langle\alpha,\alpha\rangle}\alpha\right)=w\beta-\frac{2\langle\alpha,\beta\rangle}{\langle\alpha,\alpha\rangle}w\alpha\in\Phi^-.$$
If $\beta \in I_{\Phi}(w)$, then $w \not \in C$; if $\alpha \not \in I_{\Phi}(w)$ then $w \not \in C_{\alpha}$ and we are done. So assume now that $\beta \not \in I_{\Phi}(w)$ and $\alpha \in I_{\Phi}(w)$, meaning that $w \beta \in \Phi^+$ and $w \alpha \in \Phi^-$. The fact that $ws_{\alpha}\beta \in \Phi^-$ then implies that $\langle\alpha,\beta\rangle<0$ and thus $s_{\alpha}\beta=\beta-\frac{2\langle\alpha,\beta\rangle}{\langle\alpha,\alpha\rangle}\alpha$ is a positive root greater than $\beta$ in the root poset. As $\beta\notin A$ and $A$ is an order ideal, $s_{\alpha}\beta\notin A$. However $w(s_{\alpha}\beta)\in\Phi^-$, so $s_{\alpha}\beta\in I_{\Phi}(w)$ and thus $w \not \in C=W^A$.
\end{proof}

The next lemma is purely about root systems.
\begin{lemma}\label{lem:semiorder-orderideal}
Let $J\subseteq\Phi^+$ be a nonempty order ideal. Then there exists a simple root $\alpha_i\in J$ such that there do not exist $\beta_1\neq\beta_2\in J$ with $s_i\beta_1,s_i\beta_2\in\Phi^+\setminus J$.
\end{lemma}

The proof of Lemma~\ref{lem:semiorder-orderideal} is quite technical, and will be delayed to the end of the section. Let us first see how Lemmas~\ref{lem:semiorder-uncommon} and \ref{lem:semiorder-orderideal} imply the main theorem of this section.

\begin{proof}[Proof of Theorem~\ref{thm:semiorder}]
Let $C=W^A$ where $A\subseteq\Phi^+$ is a nonempty order ideal. According to Lemma~\ref{lem:semiorder-uncommon}, it suffices to find some root $\zeta\in\Phi^+$ such that $\delta_C(\zeta)\geq\frac{1}{3}$. 

By Lemma~\ref{lem:semiorder-orderideal}, fix $\alpha\in A\cap\Delta$ such that there is at most one $\beta\in A$ such that $s_\alpha\beta\in\Phi^+\setminus A$. In the proof of Lemma~\ref{lem:semiorder-uncommon} we showed that right multiplication by $s_{\alpha}$ is an injection $C_{\alpha}\hookrightarrow C\setminus C_{\alpha}$. Let $K=C\setminus(C_{\alpha}\cup C_{\alpha}s_{\alpha})$. As $|C_{\alpha}s_{\alpha}|=|C_{\alpha}|=\delta_C(\alpha)|C|$, we have that $|K| \geq (1-2 \delta_C(\alpha))|C|$. Any $w\in K\subseteq C\setminus C_{\alpha}$ does not have $\alpha$ as an inversion, so $ws_{\alpha}$ does have $\alpha$ as an inversion. If we had $ws_{\alpha} \in C$, then it would lie in $C_{\alpha}$, and so $w$ would lie in $C_{\alpha}s_{\alpha}$, contradicting $w \in K$. Thus $ws_{\alpha}\notin C$, meaning that there is some $\gamma\in I_{\Phi}(ws_{\alpha})$ such that $\gamma\notin A$. As $\alpha\in A$, we have $\gamma\neq\alpha$ and thus $s_{\alpha}\gamma\in\Phi^+$, since $\alpha$ is the unique inversion of $s_{\alpha}$.  Since $ws_{\alpha}\gamma\in \Phi^-$ and $s_{\alpha}\gamma\in\Phi^+$ we have $s_{\alpha}\gamma\in I_{\Phi}(w)\subseteq A$. Now $s_{\alpha}\gamma\in A$ and $s_\alpha(s_{\alpha}\gamma)=\gamma\in\Phi^+\setminus A$, so $s_{\alpha}\gamma=\beta$, the unique $\beta\in A$ (now we know it exists) such that $s_{\alpha}\beta\in\Phi^+\setminus A$. This argument shows that $K\subseteq C_{\beta}$, so
\[
(1-2 \delta_C(\alpha))|C| \leq |K| \leq |C_{\beta}| = \delta_C(\beta) |C|.
\]
This implies that $\max(\delta_C(\alpha),\delta_C(\beta)) \geq \frac{1}{3}$.
\end{proof}

We now proceed to prove the remaining lemma; this is the only part of the argument which is type-dependent.

\begin{proof}[Proof of Lemma~\ref{lem:semiorder-orderideal}]
If $\Phi$ is reducible and $\alpha_i,\beta \in J$ lie in different components, then $s_{i}\beta=\beta \in J$. Thus we may assume that $\Phi$ is irreducible. By induction on the rank of $\Phi$, we may also assume that $\Delta\subseteq J$: if a simple root $\alpha_i$ does not belong to $J$, then $J$ is an order ideal in $\Phi'$, the root subsystem of $\Phi$ generated by $\Delta\setminus\{\alpha_i\}$, which is of a strictly smaller rank.

Let $H$ be the Hasse diagram of the root poset $\Phi^+$, and construct a graph $G$ on $\Phi^+$ in the following way. For two different positive roots $\beta$ and $\beta'$, we connect $\beta$ and $\beta'$ and label this edge by $\alpha\in\Delta$ if $s_{\alpha}\beta=\beta'$. Notice that each edge is uniquely labeled since if $\beta\neq\beta'$ are connected, $\beta-\beta'$ is a multiple of its label as $s_\alpha\beta=\beta-\frac{2\langle\alpha,\beta\rangle}{\langle\alpha,\alpha\rangle}\alpha$. In simply laced types ($A_{n-1}$, $D_n$, $E_6$, $E_7$, $E_8$), $\frac{2\langle\alpha,\beta\rangle}{\langle\alpha,\alpha\rangle}\in\{-1,0,1\}$ so the graph $G$ is precisely the Hasse diagram $H$ of the root poset. In types $B_n$, $C_n$, and $F_4$---where an edge with label four is present in the Coxeter diagram---if $\beta$ and $\beta'$ are connected in $G$, then their height difference is 1 or 2. 

We now consider each type separately, but the arguments are largely the same. In types $A,B,$ and $C$ we show that there are fewer than $2 |\Delta|$ edges in $G$ between $J$ and $\Phi^+ \setminus J$, so that some $\alpha_i \in \Delta$ labels at most one such edge. This is the simple root $\alpha_i$ from the lemma.

\textbf{Type} $A_{n-1}$. The Hasse diagram $H$, which is also the graph $G$, is shown in Figure~\ref{fig:typeA-rootposet}. Drawing the root $e_i-e_j$ at coordinate $(i+j-3,j-i-1)$, the graph $G$ has a grid structure as indicated by the figure. An order ideal $J\supset\Delta$ can be naturally associated to path $P$ from $(0,1)$ to $(2n-4,1)$ with steps $(1,1)$ or $(1,-1)$ that never goes below the $x$-axis, such that $J$ is exactly the set of elements below $P$. Edges in $G$ between $J$ and $\Phi^+\setminus J$ are exactly those that intersect $P$. Each step of $P$ crosses at most one edge in the Hasse diagram; thus there can be at most $2n-4< 2 |\Delta| = 2n-2$ edges between $J$ and $\Phi^+\setminus J$.

\begin{figure}[ht!]
\centering
\begin{tikzpicture}[scale=0.7]
\node at (0,0) {$\bullet$};
\node[below] at (0,0) {\tiny$e_{1}{-}e_{2}$};
\node at (1,1) {$\bullet$};
\node[below] at (1,1) {\tiny$e_{1}{-}e_{3}$};
\node at (2,0) {$\bullet$};
\node[below] at (2,0) {\tiny$e_{2}{-}e_{3}$};
\node at (2,2) {$\bullet$};
\node[below] at (2,2) {\tiny$e_{1}{-}e_{4}$};
\node at (3,1) {$\bullet$};
\node[below] at (3,1) {\tiny$e_{2}{-}e_{4}$};
\node at (4,0) {$\bullet$};
\node[below] at (4,0) {\tiny$e_{3}{-}e_{4}$};
\node at (3,3) {$\bullet$};
\node[below] at (3,3) {\tiny$e_{1}{-}e_{5}$};
\node at (4,2) {$\bullet$};
\node[below] at (4,2) {\tiny$e_{2}{-}e_{5}$};
\node at (5,1) {$\bullet$};
\node[below] at (5,1) {\tiny$e_{3}{-}e_{5}$};
\node at (6,0) {$\bullet$};
\node[below] at (6,0) {\tiny$e_{4}{-}e_{5}$};
\node at (4,4) {$\bullet$};
\node[below] at (4,4) {\tiny$e_{1}{-}e_{6}$};
\node at (5,3) {$\bullet$};
\node[below] at (5,3) {\tiny$e_{2}{-}e_{6}$};
\node at (6,2) {$\bullet$};
\node[below] at (6,2) {\tiny$e_{3}{-}e_{6}$};
\node at (7,1) {$\bullet$};
\node[below] at (7,1) {\tiny$e_{4}{-}e_{6}$};
\node at (8,0) {$\bullet$};
\node[below] at (8,0) {\tiny$e_{5}{-}e_{6}$};
\node at (5,5) {$\bullet$};
\node[below] at (5,5) {\tiny$e_{1}{-}e_{7}$};
\node at (6,4) {$\bullet$};
\node[below] at (6,4) {\tiny$e_{2}{-}e_{7}$};
\node at (7,3) {$\bullet$};
\node[below] at (7,3) {\tiny$e_{3}{-}e_{7}$};
\node at (8,2) {$\bullet$};
\node[below] at (8,2) {\tiny$e_{4}{-}e_{7}$};
\node at (9,1) {$\bullet$};
\node[below] at (9,1) {\tiny$e_{5}{-}e_{7}$};
\node at (10,0) {$\bullet$};
\node[below] at (10,0) {\tiny$e_{6}{-}e_{7}$};
\draw(0,0)--(0,0)--(5,5);
\draw(1,1)--(2,0)--(6,4);
\draw(2,2)--(4,0)--(7,3);
\draw(3,3)--(6,0)--(8,2);
\draw(4,4)--(8,0)--(9,1);
\draw(5,5)--(10,0)--(10,0);

\draw[dashed](0,1)--(1,2)--(2,1)--(3,0)--(4,1)--(5,2)--(6,3)--(7,2)--(8,3)--(10,1);
\end{tikzpicture}
\caption{The Hasse diagram of the root poset of type $A_{6}$ with a dashed line indicating an order ideal.}
\label{fig:typeA-rootposet}
\end{figure}

\textbf{Types} $B_n$ \textbf{and} $C_n$, $n\geq2$. Similarly to the case of type $A_{n-1}$, the Hasse diagram of the root poset has a grid structure so that for each order ideal $J$ there is an associated lattice path $P$ from $(0,1)$ to $(2n-2,y_0)$ for some integer $y_0$ with steps $(1,1)$ or $(1,-1)$ such that $J$ consists precisely of those roots below the path $P$. This is achieved by drawing the roots at the following positions in the plane.
\begin{itemize}
    \item Type $B_n$:
    \begin{itemize}
        \item[-] Draw $e_i-e_j$ at coordinate $(i+j-3,j-i-1)$.
        \item[-] Draw $e_i$ at coordinate $(n+i-2,n-i)$.
        \item[-] Draw $e_i+e_j$ at coordinate $(2n+i-j-1,2n+1-i-j)$.
    \end{itemize}
    \item Type $C_n$:
    \begin{itemize}
        \item[-] Draw $e_i-e_j$ at coordinate $(i+j-3,j-i-1)$.
        \item[-] Draw $2e_i$ at coordinate $(2n-2,2n-2i)$.
        \item[-] Draw $e_i+e_j$ at coordinate $(2n+i-j-2,2n-i-j)$.
    \end{itemize}
\end{itemize}

Although the graphs $G$ and $H$ are not the same in this case, the difference is minor and is shown in Figure~\ref{fig:typeBC-rootposet}. Crucially, each edge of $G$ is parallel to $(1,1)$ or $(1,-1)$, except for the vertical edges on the far right-hand-side in type $C$. Since the vertices of $G$ lie in $\mathbb{Z}^2$, this implies that each step of $P$ again crosses only one edge of $G$. Thus there are at most $2n-2 < 2 |\Delta|=2n$ edges in $G$ between $J$ and $\Phi^+\setminus J$.
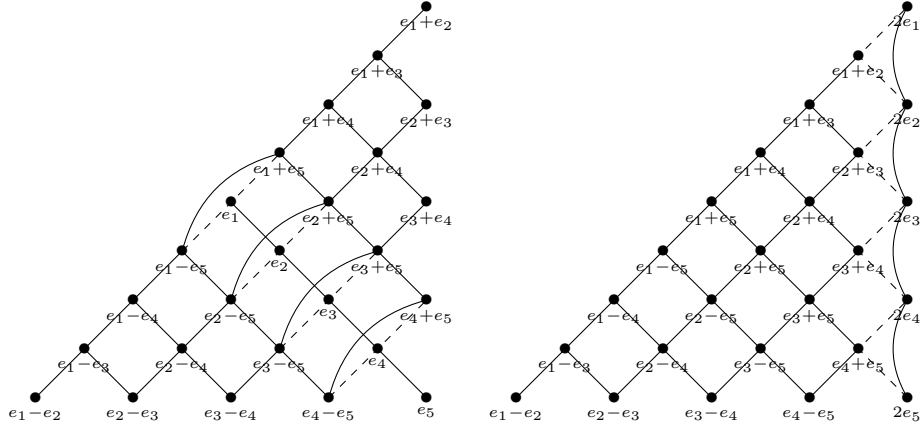
\begin{figure}[ht!]
\centering
\begin{tikzpicture}[scale=0.65]
\node at (0,0) {$\bullet$};
\node[below] at (0,0) {\tiny$e_{1}{-}e_{2}$};
\node at (8,8) {$\bullet$};
\node[below] at (8,8) {\tiny$e_{1}{+}e_{2}$};
\node at (1,1) {$\bullet$};
\node[below] at (1,1) {\tiny$e_{1}{-}e_{3}$};
\node at (7,7) {$\bullet$};
\node[below] at (7,7) {\tiny$e_{1}{+}e_{3}$};
\node at (2,0) {$\bullet$};
\node[below] at (2,0) {\tiny$e_{2}{-}e_{3}$};
\node at (8,6) {$\bullet$};
\node[below] at (8,6) {\tiny$e_{2}{+}e_{3}$};
\node at (2,2) {$\bullet$};
\node[below] at (2,2) {\tiny$e_{1}{-}e_{4}$};
\node at (6,6) {$\bullet$};
\node[below] at (6,6) {\tiny$e_{1}{+}e_{4}$};
\node at (3,1) {$\bullet$};
\node[below] at (3,1) {\tiny$e_{2}{-}e_{4}$};
\node at (7,5) {$\bullet$};
\node[below] at (7,5) {\tiny$e_{2}{+}e_{4}$};
\node at (4,0) {$\bullet$};
\node[below] at (4,0) {\tiny$e_{3}{-}e_{4}$};
\node at (8,4) {$\bullet$};
\node[below] at (8,4) {\tiny$e_{3}{+}e_{4}$};
\node at (3,3) {$\bullet$};
\node[below] at (3,3) {\tiny$e_{1}{-}e_{5}$};
\node at (5,5) {$\bullet$};
\node[below] at (5,5) {\tiny$e_{1}{+}e_{5}$};
\node at (4,2) {$\bullet$};
\node[below] at (4,2) {\tiny$e_{2}{-}e_{5}$};
\node at (6,4) {$\bullet$};
\node[below] at (6,4) {\tiny$e_{2}{+}e_{5}$};
\node at (5,1) {$\bullet$};
\node[below] at (5,1) {\tiny$e_{3}{-}e_{5}$};
\node at (7,3) {$\bullet$};
\node[below] at (7,3) {\tiny$e_{3}{+}e_{5}$};
\node at (6,0) {$\bullet$};
\node[below] at (6,0) {\tiny$e_{4}{-}e_{5}$};
\node at (8,2) {$\bullet$};
\node[below] at (8,2) {\tiny$e_{4}{+}e_{5}$};
\node at (4,4) {$\bullet$};
\node[below] at (4,4) {\tiny$e_1$};
\node at (5,3) {$\bullet$};
\node[below] at (5,3) {\tiny$e_2$};
\node at (6,2) {$\bullet$};
\node[below] at (6,2) {\tiny$e_3$};
\node at (7,1) {$\bullet$};
\node[below] at (7,1) {\tiny$e_4$};
\node at (8,0) {$\bullet$};
\node[below] at (8,0) {\tiny$e_5$};
\draw(0,0)--(0,0)--(3,3);
\draw(8,8)--(8,8)--(5,5);
\draw(1,1)--(2,0)--(4,2);
\draw(7,7)--(8,6)--(6,4);
\draw(2,2)--(4,0)--(5,1);
\draw(6,6)--(8,4)--(7,3);
\draw(3,3)--(6,0)--(6,0);
\draw(5,5)--(8,2)--(8,2);
\draw(4,4)--(8,0);
\draw[dashed](3,3)--(5,5);
\draw(3,3) to[bend left] (5,5);
\draw[dashed](4,2)--(6,4);
\draw(4,2) to[bend left] (6,4);
\draw[dashed](5,1)--(7,3);
\draw(5,1) to[bend left] (7,3);
\draw[dashed](6,0)--(8,2);
\draw(6,0) to[bend left] (8,2);
\end{tikzpicture}
\ 
\begin{tikzpicture}[scale=0.65]
\node at (0,0) {$\bullet$};
\node[below] at (0,0) {\tiny$e_{1}{-}e_{2}$};
\node at (7,7) {$\bullet$};
\node[below] at (7,7) {\tiny$e_{1}{+}e_{2}$};
\node at (1,1) {$\bullet$};
\node[below] at (1,1) {\tiny$e_{1}{-}e_{3}$};
\node at (6,6) {$\bullet$};
\node[below] at (6,6) {\tiny$e_{1}{+}e_{3}$};
\node at (2,0) {$\bullet$};
\node[below] at (2,0) {\tiny$e_{2}{-}e_{3}$};
\node at (7,5) {$\bullet$};
\node[below] at (7,5) {\tiny$e_{2}{+}e_{3}$};
\node at (2,2) {$\bullet$};
\node[below] at (2,2) {\tiny$e_{1}{-}e_{4}$};
\node at (5,5) {$\bullet$};
\node[below] at (5,5) {\tiny$e_{1}{+}e_{4}$};
\node at (3,1) {$\bullet$};
\node[below] at (3,1) {\tiny$e_{2}{-}e_{4}$};
\node at (6,4) {$\bullet$};
\node[below] at (6,4) {\tiny$e_{2}{+}e_{4}$};
\node at (4,0) {$\bullet$};
\node[below] at (4,0) {\tiny$e_{3}{-}e_{4}$};
\node at (7,3) {$\bullet$};
\node[below] at (7,3) {\tiny$e_{3}{+}e_{4}$};
\node at (3,3) {$\bullet$};
\node[below] at (3,3) {\tiny$e_{1}{-}e_{5}$};
\node at (4,4) {$\bullet$};
\node[below] at (4,4) {\tiny$e_{1}{+}e_{5}$};
\node at (4,2) {$\bullet$};
\node[below] at (4,2) {\tiny$e_{2}{-}e_{5}$};
\node at (5,3) {$\bullet$};
\node[below] at (5,3) {\tiny$e_{2}{+}e_{5}$};
\node at (5,1) {$\bullet$};
\node[below] at (5,1) {\tiny$e_{3}{-}e_{5}$};
\node at (6,2) {$\bullet$};
\node[below] at (6,2) {\tiny$e_{3}{+}e_{5}$};
\node at (6,0) {$\bullet$};
\node[below] at (6,0) {\tiny$e_{4}{-}e_{5}$};
\node at (7,1) {$\bullet$};
\node[below] at (7,1) {\tiny$e_{4}{+}e_{5}$};
\node at (8,8) {$\bullet$};
\node[below] at (8,8) {\tiny$2e_1$};
\node at (8,6) {$\bullet$};
\node[below] at (8,6) {\tiny$2e_2$};
\node at (8,4) {$\bullet$};
\node[below] at (8,4) {\tiny$2e_3$};
\node at (8,2) {$\bullet$};
\node[below] at (8,2) {\tiny$2e_4$};
\node at (8,0) {$\bullet$};
\node[below] at (8,0) {\tiny$2e_5$};
\draw(0,0)--(0,0)--(3,3);
\draw(7,7)--(7,7)--(4,4);
\draw(1,1)--(2,0)--(4,2);
\draw(6,6)--(7,5)--(5,3);
\draw(2,2)--(4,0)--(5,1);
\draw(5,5)--(7,3)--(6,2);
\draw(3,3)--(6,0)--(6,0);
\draw(4,4)--(7,1)--(7,1);
\draw(3,3)--(4,4);
\draw[dashed](8,0)--(7,1)--(8,2);
\draw(8,0) to[bend left] (8,2);
\draw(4,2)--(5,3);
\draw[dashed](8,2)--(7,3)--(8,4);
\draw(8,2) to[bend left] (8,4);
\draw(5,1)--(6,2);
\draw[dashed](8,4)--(7,5)--(8,6);
\draw(8,4) to[bend left] (8,6);
\draw(6,0)--(7,1);
\draw[dashed](8,6)--(7,7)--(8,8);
\draw(8,6) to[bend left] (8,8);
\end{tikzpicture}
\caption{Hasse diagram $H$ and graph $G$ of the root poset of type $B_5$ (left) and $C_5$ (right), where edges in $H\setminus G$ are drawn in dashed lines and edges in $G\setminus H$ are curved.}
\label{fig:typeBC-rootposet}
\end{figure}

\textbf{Type} $D_n$, $n\geq4$. As type $D_n$ is simply laced, $G=H$. The depiction of the root poset of type $D_n$ in Figure~\ref{fig:typeD-rootposet} may be helpful when reading the following argument.

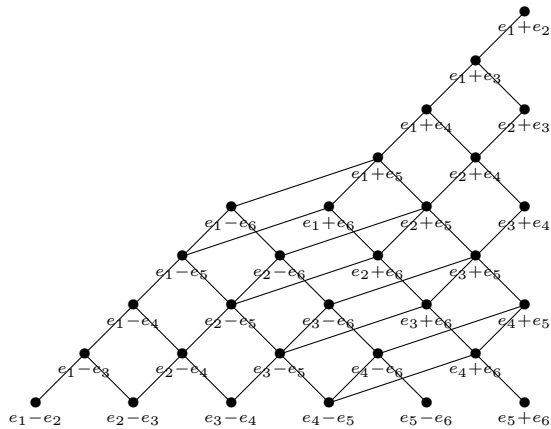
\begin{figure}[ht!]
\centering
\begin{tikzpicture}[scale=0.65]
\node at (0,0) {$\bullet$};
\node[below] at (0,0) {\tiny$e_{1}{-}e_{2}$};
\node at (10,8) {$\bullet$};
\node[below] at (10,8) {\tiny$e_{1}{+}e_{2}$};
\node at (1,1) {$\bullet$};
\node[below] at (1,1) {\tiny$e_{1}{-}e_{3}$};
\node at (9,7) {$\bullet$};
\node[below] at (9,7) {\tiny$e_{1}{+}e_{3}$};
\node at (2,0) {$\bullet$};
\node[below] at (2,0) {\tiny$e_{2}{-}e_{3}$};
\node at (10,6) {$\bullet$};
\node[below] at (10,6) {\tiny$e_{2}{+}e_{3}$};
\node at (2,2) {$\bullet$};
\node[below] at (2,2) {\tiny$e_{1}{-}e_{4}$};
\node at (8,6) {$\bullet$};
\node[below] at (8,6) {\tiny$e_{1}{+}e_{4}$};
\node at (3,1) {$\bullet$};
\node[below] at (3,1) {\tiny$e_{2}{-}e_{4}$};
\node at (9,5) {$\bullet$};
\node[below] at (9,5) {\tiny$e_{2}{+}e_{4}$};
\node at (4,0) {$\bullet$};
\node[below] at (4,0) {\tiny$e_{3}{-}e_{4}$};
\node at (10,4) {$\bullet$};
\node[below] at (10,4) {\tiny$e_{3}{+}e_{4}$};
\node at (3,3) {$\bullet$};
\node[below] at (3,3) {\tiny$e_{1}{-}e_{5}$};
\node at (7,5) {$\bullet$};
\node[below] at (7,5) {\tiny$e_{1}{+}e_{5}$};
\node at (4,2) {$\bullet$};
\node[below] at (4,2) {\tiny$e_{2}{-}e_{5}$};
\node at (8,4) {$\bullet$};
\node[below] at (8,4) {\tiny$e_{2}{+}e_{5}$};
\node at (5,1) {$\bullet$};
\node[below] at (5,1) {\tiny$e_{3}{-}e_{5}$};
\node at (9,3) {$\bullet$};
\node[below] at (9,3) {\tiny$e_{3}{+}e_{5}$};
\node at (6,0) {$\bullet$};
\node[below] at (6,0) {\tiny$e_{4}{-}e_{5}$};
\node at (10,2) {$\bullet$};
\node[below] at (10,2) {\tiny$e_{4}{+}e_{5}$};
\node at (4,4) {$\bullet$};
\node[below] at (4,4) {\tiny$e_{1}{-}e_{6}$};
\node at (6,4) {$\bullet$};
\node[below] at (6,4) {\tiny$e_{1}{+}e_{6}$};
\node at (5,3) {$\bullet$};
\node[below] at (5,3) {\tiny$e_{2}{-}e_{6}$};
\node at (7,3) {$\bullet$};
\node[below] at (7,3) {\tiny$e_{2}{+}e_{6}$};
\node at (6,2) {$\bullet$};
\node[below] at (6,2) {\tiny$e_{3}{-}e_{6}$};
\node at (8,2) {$\bullet$};
\node[below] at (8,2) {\tiny$e_{3}{+}e_{6}$};
\node at (7,1) {$\bullet$};
\node[below] at (7,1) {\tiny$e_{4}{-}e_{6}$};
\node at (9,1) {$\bullet$};
\node[below] at (9,1) {\tiny$e_{4}{+}e_{6}$};
\node at (8,0) {$\bullet$};
\node[below] at (8,0) {\tiny$e_{5}{-}e_{6}$};
\node at (10,0) {$\bullet$};
\node[below] at (10,0) {\tiny$e_{5}{+}e_{6}$};
\draw(0,0)--(0,0)--(4,4);
\draw(10,8)--(10,8)--(6,4);
\draw(1,1)--(2,0)--(5,3);
\draw(9,7)--(10,6)--(7,3);
\draw(2,2)--(4,0)--(6,2);
\draw(8,6)--(10,4)--(8,2);
\draw(3,3)--(6,0)--(7,1);
\draw(7,5)--(10,2)--(9,1);
\draw(4,4)--(8,0)--(8,0);
\draw(6,4)--(10,0)--(10,0);
\draw(4,4)--(7,5);
\draw(3,3)--(6,4);
\draw(5,3)--(8,4);
\draw(4,2)--(7,3);
\draw(6,2)--(9,3);
\draw(5,1)--(8,2);
\draw(7,1)--(10,2);
\draw(6,0)--(9,1);
\end{tikzpicture}
\caption{The Hasse diagram of the root poset of type $D_6$.  Edges with label $e_{n-1}+e_n$ are precisely the longer edges with slope not in $\{\pm1\}$. }
\label{fig:typeD-rootposet}
\end{figure}

Suppose that $e_{n-2}+e_n \not \in J$. Then all roots of the form $e_i+e_j$ are not in $J$ except the simple root $e_{n-1}+e_n$. Let $\Delta' = \Delta \setminus \{e_{n-1}+e_n\}$. All edges in $G$ with labels in $\Delta'$ that can possibly go between $J$ and $\Phi^+\setminus J$ are edges of the form $e_i-e_j<e_{i'}-e_{j'}$ or $e_{n-1}+e_n<e_{n-2}+e_n$. This can be seen easily from Figure~\ref{fig:typeD-rootposet} if we ignore all edges labeled $e_{n-1}+e_n$. These edges are precisely the type $A_{n-1}$ edges, plus an additional edge corresponding to $e_{n-1}+e_n<e_{n-2}+e_n$. Using the type $A$ argument, we see that there are at most $(2n-4)+1=2n-3< 2|\Delta'|=2n-2$ such edges, so we are done in this case.

By the previous paragraph, we may assume that $e_{n-2}+e_n \in J$. Since there is an automorphism of the root poset induced by the automorphism of the Coxeter diagram exchanging $e_{n-1}-e_n$ and $e_{n-1}+e_n$, we may also assume that $e_{n-2}-e_n \in J$. The counting arguments used above cannot work in this case, and we need to pay attention to labels of edges. For each simple root $\alpha\in\Delta$, let 
\[
B(\alpha) \coloneqq \{\beta\in J\:|\: \alpha+\beta\in\Phi^+\setminus J\}.
\]
We will show that there is some $\alpha$ with $|B(\alpha)| \leq 1$. 

Suppose now that for some $k\leq n-2$, the set $B(e_k-e_{k+1})$ contains at least two roots of the form $e_i-e_k$ with $i<k$, say $e_i-e_k$ and $e_{i'}-e_k$, with $i<i'$. Suppose also that we have chosen the smallest such $k$. Then $e_i-e_k,e_{i'}-e_k\in J$ while $e_i-e_{k+1},e_{i'}-e_{k+1}\notin J$. Consider the simple root $\alpha=e_i-e_{i+1}$ and suppose $\beta\in B(\alpha)$. If $\beta=e_j+e_{i+1}$, then as $i+1\leq i'$ we have $\beta\geq e_{i'}-e_{k+1}\notin J$, so $\beta\notin J$, contrary to our assumption. If instead $\beta=e_{i+1}-e_j$ for some $j>i+1$, and if $j\leq k$, then $\alpha+\beta=e_{i}-e_j\leq e_i-e_k\in J$. This is again impossible; the case $j\geq k+1$ is similar. As a result, every root $\beta\in B(\alpha)$ must be of the form $e_j-e_i$ for $j<i$. Since $k$ is minimal, we must have that $|B(\alpha)| \leq 1$, and we are done. Therefore, we may assume in the remainder that for all $k \leq n-2$, $B(e_k-e_{k+1})$ contains at most one root of the form $e_i-e_k$ with $i<k$.

Consider the simple root $e_{n-2}-e_{n-1}$; since $e_{n-2}-e_n\in J$ we see that $e_{n-1}-e_n\notin B(e_{n-2}-e_{n-1})$. By the paragraph above, there is at most one root in $B(e_{n-2}-e_{n-1})$ of the form $e_i-e_j$. Either $|B(e_{n-2}-e_{n-1})| \leq 1$ and we are done, or it contains a root of the form $e_i+e_j$, which must in fact be of the form $e_i+e_{n-1}$ with $i<n-2$. In the later case, we see that $e_i+e_{n-1} \geq e_{n-3}+e_{n-1} \in J$. Now suppose that $e_k+e_{k+2}\in J$ for some $k>1$. Consider the simple root $\alpha=e_k-e_{k+1}$ and a root $\beta\in B(\alpha)$ of the form $\beta=e_i-e_j$. If $i=k+1$, then $\alpha+\beta=e_k-e_{j}\leq e_k+e_{k+2}\in J$ so $\alpha+\beta\in J$ which is not possible. There is at most one $\beta\in B(\alpha)$ of the form $e_i-e_k$. Either $|B(\alpha)| \leq 1$ and we are done, or there is some $e_i+e_{k+1}\in B(\alpha)$. In the latter case, if $i>k+1$, then
\[
(e_i+e_{k+1})+(e_k-e_{k+1})=e_i+e_{k}\leq e_{k}+e_{k+2}\in J,
\]
which is impossible. As a result, $i<k$, and $e_i+e_{k+1} \geq e_{k-1}+e_{k+1}\in J$. We have shown by induction that $e_k+e_{k+2}\in J$ for all $k\geq1$. This finishes the proof, since it implies that $J$ contains every positive root except the highest one, making $B(\alpha)=\emptyset$ for all $\alpha\in\Delta$ except possibly $\alpha=e_2-e_3$. 

\textbf{Exceptional types}. We deal with the exceptional types $E_6, E_7, E_8, F_4, G_2$ via a direct computer search, using SageMath. This search was implemented by iterating through all order ideals $J$ in the root poset, simple roots $\alpha_i$, and other positive roots $\beta \in J$ and counting occurrences of $s_i \beta \in \Phi^+ \setminus J$. The simple roots, positive roots, root poset and its order ideals, and the action of the Weyl group on the roots have all been implemented as methods in the \texttt{RootSystem} class in SageMath, for which we are very grateful to its developers.

There are 25080 order ideals in the root poset of type $E_8$, for which we were able to verify the lemma in 20 minutes on a personal computer. The other exceptional types each took under one minute to check.
\end{proof}

It would be interesting to find a type-independent proof of Lemma~\ref{lem:semiorder-orderideal}, as this would yield a uniform proof of Theorem~\ref{thm:semiorder}.

\section{A uniform bound for finite Weyl groups}
\label{sec:bounds-Weyl}

In this section we provide a uniform lower bound for the balance constant of any non-singleton convex subset in any finite Weyl group; in the case of the symmetric group such a constant bound away from zero was first established by Kahn and Saks \cite{Kahn-Saks}. 

\begin{theorem}\label{thm:uniform-bound}
There exists an absolute constant $\epsilon>0$ such that for any non-singleton convex set $C$ in any finite Weyl group we have
\[
b(C)>\epsilon.
\]
\end{theorem}
Theorem~\ref{thm:uniform-bound} will be proven as Theorem~\ref{thm:uniform-bound-specific} after more notation has been introduced. The bounds for each irreducible type are given in Table~\ref{tab:centroid-data}. In particular, we can take $\epsilon=1/2e^{12}$ as a uniform bound. However, the bound for classical types are much better: $1/2e$ for type $A_n$ (obtained by Kahn and Linial~\cite{Kahn-Linial}), $1/2e^2$ for types $B_n$ and $C_n$, and $1/2e^4$ for type $D_n$. 

Our proof of Theorem~\ref{thm:uniform-bound} is type-independent and uses a geometric argument inspired by that of Kahn and Linial \cite{Kahn-Linial}. In Section~\ref{sec:brunn-minkowski} we apply the Brunn--Minkowski Theorem from convex geometry to obtain useful bounds for general polytopes and in Section~\ref{sec:order-polytopes-bound} we prove Theorem~\ref{thm:uniform-bound} by applying these bounds to \emph{generalized order polytopes}, which we introduce. 

In Section~\ref{sec:short-root-order-polytope}, we provide a special treatment for type $B_n$, where a curious ``short-root order polytope" exists, allowing for improved bounds.

\subsection{An important lemma}

We start with a lemma that will not be used until Section~\ref{sec:short-root-order-polytope}, but the techniques will be useful for later sections. 

\begin{lemma}\label{lem:select-root}
Let $C\subseteq W(\Phi)$ be a non-singleton convex set. For $\beta\in\Phi$, define
$$h(\beta)\coloneqq\frac{1}{|C|}\sum_{w\in C}\langle w\beta,\omega_1^{\vee}+\cdots+\omega_r^{\vee}\rangle.$$
Then there exists $\beta\in\Phi^+$ such that $|h(\beta)|<1$.
\end{lemma}
\begin{proof}
Suppose that $|h(\beta)|\geq1$ for all $\beta\in\Phi$. We prove the contrapositive of the lemma by showing that $C$ is a singleton.

The symmetric bilinear form $\langle \cdot , \cdot \rangle$ is $W$-invariant, so
\begin{align*}
h(\beta)\coloneqq&\frac{1}{|C|}\sum_{w\in C}\langle w\beta,\omega_1^{\vee}+\cdots+\omega_r^{\vee}\rangle\\
=&\frac{1}{|C|}\sum_{w\in C}\left\langle \beta,w^{-1}(\omega_1^{\vee}+\cdots+\omega_r^{\vee})\right\rangle=\langle \beta,v\rangle
\end{align*}
for some $v\in E$, where $E$ is the ambient Euclidean space. Consider the set \[
\Phi_{v}^{+}\coloneqq\{\beta\in\Phi\:|\:\langle\beta,v\rangle>0\},
\]
which is another choice of positive roots for $\Phi$, with a corresponding set of simple roots $\Delta_v=\{\gamma_1,\ldots,\gamma_r\}$. There is a highest root $\xi_v\in\Phi^+_v$ with respect to this choice of positive roots. By our assumption, $h(\gamma_i)\geq1$ for $i=1,\ldots,r$, so $h(\xi_v)\geq\height(\Phi)$. However, $$h(\beta)=\frac{1}{C}\sum_{w\in C}\height(w\beta)\leq \height(\Phi)$$
for all $\beta\in\Phi$. As a result, we must have that $h(\xi_v)=\height(\Phi)$, which implies $h(\gamma_i)=1$ for all $i=1,\ldots,r$. At the same time, we must have $w\xi_v=\xi$ for all $w\in C$. 

Let $\height_v$ be the height function with respect to $\Phi^+_v$; that is, 
\[
\height_v(\beta)=c_1+\cdots+c_r
\]
if $\beta=c_1\gamma_1+\cdots+c_r\gamma_r$. In fact we must have $\height_v=h$, as both are linear functions on $E$ and agree on the basis $\Delta_v$. We will now show that $\height(w\beta)=\height_v(\beta)$ for all $\beta\in\Phi^+_v$ and $w\in C$, by reverse induction on $k=\height_v(\beta)$. The base case $k=\height(\Phi)$, $\beta=\xi_v$ has already been established. Suppose that we are done with $k$ and let us take $\beta\in\Phi_v^+$ with $\height_v(\beta)=k-1$. Let there be $M$ roots $\Phi^{\geq k}\coloneqq\{\beta_1,\ldots,\beta_M\}$ with height at least $k$ and let $\Phi^{\geq k}_v=\{\beta_1',\ldots,\beta_M'\}$ be the $M$ roots with $\height_v$ at least $k$. We may view each $w\in C$ as a permutation of the roots. By the induction hypothesis, $w$ restricts to a bijection from $\Phi_v^{\geq k}$ to $\Phi^{\geq k}$. Since $\beta\notin \Phi_v^{\geq k}$ we have $w\beta\notin\Phi^{\geq k}$ and so $\height(w\beta)\leq k-1$. Therefore, $$h(\beta)=\frac{1}{C}\sum_{w\in C}\height(w\beta)\leq k-1$$
but we know $h(\beta)=\height_v(\beta)=k-1$. This implies that $\height(w\beta)=k-1$ as desired. Thus we see that for every $w\in C$, $w$ sends $\Phi_v^+$ to $\Phi^+$. By the one-to-one correspondence (see \cite[\S 1.4]{Humphreys}) between elements of the Weyl group and choices of positive roots, there is a unique $w$ sending $\Phi_v^+$ to $\Phi^+$, so $C$ is a singleton.
\end{proof}

\subsection{A general bound for convex bodies via Brunn--Minkowski}\label{sec:brunn-minkowski}
Our arguments in this section are inspired by those of Kahn and Linial \cite{Kahn-Linial}, with some generalization required for application to general Weyl groups. 

For a convex body $Q\subseteq\R^n$ of full dimension, and a vector $v$ in $\R^n$, define $Q^+_v\coloneqq\{x\in Q\:|\: \langle v,x\rangle\geq0\}$ and $Q_v^-\coloneqq\{x\in Q\:|\:\langle v,x\rangle\leq0\}$, the two pieces of $Q$ split by the hyperplane orthogonal to $v$. Let 
\[
Q_{v}^{\lambda}\coloneqq\{x\in Q\:|\:\langle v,x\rangle=\lambda\},
\]
an $(n-1)$-dimensional convex body. We start with a lemma that follows directly from the classical Brunn--Minkowski theorem.

\begin{lemma}\label{lem:Brunn-Minkowski}
The function $\lambda\mapsto \Vol(Q_v^{\lambda})^{\frac{1}{n-1}}$ is concave on $\R$. 
\end{lemma}
\begin{proof}
Take any $x,y\in\R$, $0\leq t\leq 1$ and let $z=tx+(1-t)y$. Since $Q$ is a convex body, $Q_v^z$ contains the Minkowski sum $tQ_v^x+(1-t)Q_v^y$. Then by the Brunn--Minkowski theorem, 
$$\Vol(Q_v^z)^{\frac{1}{n-1}}\geq\Vol(tQ_v^x+(1-t)Q_v^y)^{\frac{1}{n-1}}\geq t\Vol(Q_v^x)^{\frac{1}{n-1}}+(1-t)\Vol(Q_v^y)^{\frac{1}{n-1}}.$$
This establishes the concavity of the function $\lambda\mapsto\Vol(Q_v^{\lambda})^{\frac{1}{n-1}}.$
\end{proof}

\begin{prop}\label{prop:bm}
Let $Q\subseteq\R^n$ be a full-dimensional compact convex body with centroid $c_Q$. Let $m\geq0$ and $v\in\R^n$ be such that $\langle v,c_Q\rangle\geq\frac{-m}{n+1}$. Suppose that $\min_{x\in Q}\langle v,x\rangle=-1$ and $\max_{y\in Q}\langle v,y\rangle\geq\epsilon>0$. Then 
\[
\frac{\Vol(Q_v^+)}{\Vol(Q)}\geq
\begin{cases}
\displaystyle\frac{\epsilon}{(\epsilon+1) e^{1+(m-1)/\epsilon}},&m\geq1, \\
\displaystyle\frac{\epsilon+1-m}{(\epsilon+1)e},&0\leq m<1.
\end{cases}
\]
\end{prop}
\begin{proof}
Notice that $v$ and $Q$ can be rotated together and can be scaled by inverse factors while preserving the hypotheses. Thus $v$ can be taken to be $(1,0,\ldots,0)$. Consider the concave function
\[
f(x_1)=\Vol(Q_v^{x_1})^{\frac{1}{n-1}}
\]
in the variable $x_1$ and its graph in the plane; the function $f$ is supported on $[-1, u']$ for some $u'\geq\epsilon>0$. 

We now construct a certain piecewise linear function $g$ (see Figure~\ref{fig:double-cone}). The function $g$ is supported on $[-1,u]$ and consists of two linear segments connecting vertices $(-1,0)$, $(a,g(a))$, and $(u,0)$. Here $u \geq 0$ is determined by requiring that $g(0)=f(0)$ and by the condition
\begin{equation}
\label{eq:f-g-integral1}
\int_{0}^{u'}f(x_1)^{n-1}dx_1=\int_{0}^u g(x_1)^{n-1}dx_1. 
\end{equation}
By the concavity of $f$, the line segment connecting $(0,f(0))$ and $(u',0)$ lies below $f$, so $u\geq u'\geq\epsilon$. Next, $a \leq 0$ is determined by requiring that 
\begin{equation}
\label{eq:f-g-integral2}
\int_{-1}^0 f(x_1)^{n-1}dx_1=\int_{-1}^0g(x_1)^{n-1}dx_1.
\end{equation}
A valid choice of $a$ exists because of the concavity of $f$: the line segment connecting $(-1,0)$ and $(0,f(0))$ lies below $f$, and, for $x_1 \leq 0$, $f(x_1)$ lies below the line connecting $(0,f(0))$ and $(u,0)$. Thus, by continuity, the integrals are equal for some $-1 \leq a \leq 0$.

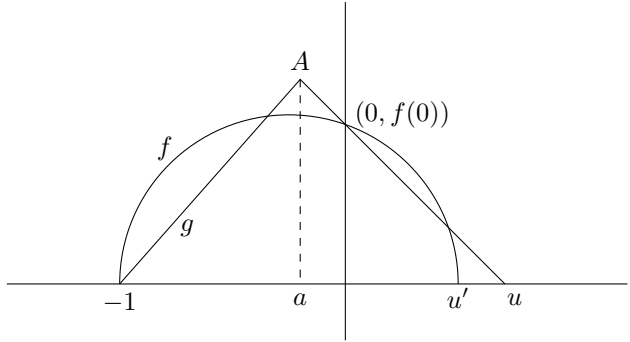
\begin{figure}[ht!]
\centering
\begin{tikzpicture}[scale=1.5]
\draw (-2.5,0)--(2,0);
\node[right] at (2.1,0) {$x_1$};
\draw(0,-0.5)--(0,2.5);
\draw (1,0) arc (0:180:1.5);
\draw (-2,0)--(-0.4,1.8142)--(1.4142,0);
\draw[dashed] (-0.4,1.8142)--(-0.4,0);
\node[above] at (-0.4,1.8142) {};
\node[below] at (1,0.05) {$u'$};
\node[below] at (1.5,0) {$u$};
\node[right] at (0,1.5) {$(0,f(0))$};
\node[below] at (-2,0) {$-1$};
\node at (-1.6,1.2) {$f$};
\node at (-1.4,0.5) {$g$};
\node[below] at (-0.4,0) {$a$};
\end{tikzpicture}
\caption{The double cone construction for Proposition~\ref{prop:bm}.  This figure is very similar to Figure 1 from \cite{Kahn-Linial}, which represents a special case.}
\label{fig:double-cone}
\end{figure}

We now construct a double cone $D$ from the function $g$. To do this, pick any compact convex body $H$ in the hyperplane $x_1=a$ which contains $(a,0,\ldots,0)$ and has $(n-1)$-dimensional volume $g(a)^{n-1}$. Let $D_1$ be the cone over $H$ with apex $(-1,0,\ldots,0)$ and $D_2$ be the cone over $H$ with apex $(u,0,\ldots,0)$; let $D=D_1\cup D_2$. It is then clear that $g(x_1)=\Vol(D_v^{x_1})^{\frac{1}{n-1}}$. By construction of $g$, we have $\Vol(D^+_v)=\Vol(Q^+_v)$ and $\Vol(D^-_v)=\Vol(Q^-_v)$. Writing $h_1=a+1>0$ and $h_2=u-a>0$ for the heights of the cones $D_1$ and $D_2$, with $h_1+h_2=u+1$, we have 
\[
\Vol(D_1)/\Vol(D_2)=h_1/h_2.
\]

For some $-1 \leq a' \leq 0$, we have $f(x_1) \geq g(x_1)$ for $-1 \leq x_1 \leq a'$ and $g(x_1) \geq f(x_1)$ for $a' \leq x_1 \leq 0$. By (\ref{eq:f-g-integral2}), this implies that 
\begin{equation}
\label{eq:f-g-integral3}
\int_{-1}^t f(x_1)^{n-1}dx_1\geq \int_{-1}^t g(x_1)^{n-1}dx_1
\end{equation}
for all $-1\leq t\leq 0$. A similar argument for positive $x_1$ using (\ref{eq:f-g-integral1}) shows that (\ref{eq:f-g-integral3}) in fact holds for all $-1 \leq t \leq u$, with equality at $u$. This tells us that $(c_D)_1 \geq (c_Q)_1 \geq\frac{-m}{n+1}$, where we write $(p)_1$ for the $x_1$-coordinate of $p \in \mathbb{R}^n$. 

The ratio of interest is
\[
\frac{\Vol(Q_v^+)}{\Vol(Q)}=\frac{\Vol(D_v^+)}{\Vol(D)}=\frac{\Vol(D_v^+)}{\Vol(D_2)}\frac{\Vol(D_2)}{\Vol(D)}=\left(\frac{u}{h_2}\right)^{n}\frac{h_2}{h_1+h_2}=\left(\frac{u}{h_2}\right)^{n-1}\frac{u}{u+1}.
\]

We compute
\begin{align*}
(c_{D_1})_1&= \frac{\int_{-1}^a x \left(g(a) \frac{x+1}{a+1}\right)^{n-1} dx}{\int_{-1}^a \left(g(a) \frac{x+1}{a+1}\right)^{n-1} dx} \\
&=\frac{-1+na}{n+1} 
\end{align*}
and similarly $(c_{D_2})_1=\frac{u+na}{n+1}$. Thus
\begin{align*}
-\frac{m}{n+1} & \leq (c_Q)_1  \leq (c_D)_1 \\
&= \frac{1}{h_1+h_2}\left(h_1\frac{-1+na}{n+1}+h_2\frac{u+na}{n+1}\right) \\
&=\frac{nu-1-(n-1)h_2}{n+1}.
\end{align*}

Rearranging, we obtain 
\begin{align*}
h_2&\leq\frac{nu+m-1}{n-1}, \\
\frac{u}{h_2}&\geq \frac{n-1}{n+(m-1)/u}.
\end{align*}

Before the final step, note that $e^x\geq 1+x$ for all $x\in\R$, so $e^{x/n}\geq 1+x/n$ and $e^x\geq(1+x/n)^n$ and $(\frac{n}{n+x})^n\geq e^{-x}$ for all $x\geq -n$. We now have that
\[
\frac{\Vol(Q_v^+)}{\Vol(Q)}=\left(\frac{u}{h_2}\right)^{n-1}\frac{u}{u+1} \geq \left(\frac{n-1}{n+(m-1)/u}\right)^{n-1}\frac{u}{u+1}\geq e^{-1-(m-1)/u}\frac{u}{u+1}.
\]
Recall that $u\geq u'\geq \epsilon$. Thus, if $m\geq1$, then
\[
\frac{\Vol(Q_v^+)}{\Vol(Q)}\geq\frac{\epsilon}{(\epsilon+1) e^{1+(m-1)/\epsilon}},
\]
and if $m\leq1$, then
\[
\frac{\Vol(Q_v^+)}{\Vol(Q)}\geq e^{-1}\left(1-\frac{m-1}{u}\right)\frac{u}{u+1}
=e^{-1}\frac{u+1-m}{u+1}\geq\frac{\epsilon+1-m}{(\epsilon+1)e},
\]
as desired.
\end{proof}

The following corollary to Proposition~\ref{prop:bm} is more useful in practice.

\begin{cor}\label{cor:bm}
Let $Q\subseteq\R^n$ be a full-dimensional compact convex body with centroid $c_Q$. Let $m\geq1$ and $v\in\R^n$ be such that $\langle v,c_Q\rangle\geq\frac{-m}{n+1}$. Suppose that $\min_{x\in Q}\langle v,x\rangle\leq-1$ and $\max_{y\in Q}\langle v,y\rangle\geq1$. Then
$$\frac{\Vol(Q_v^+)}{\Vol(Q)}\geq\frac{1}{2e^m}.$$
\end{cor}
\begin{proof}
Let $\min_{x\in Q}\langle v,x\rangle=-z$ with $z\geq1$. Scale $Q$ by a factor of $1/z$ to obtain $Q'$. Then $\min_{x\in Q'}\langle v,x\rangle=-1$, $\max_{y\in Q'}\langle v,y\rangle \geq 1/z$, and $\langle v,c_{Q'}\rangle\geq\frac{-m/z}{n+1}$. Note that for $x\geq1$ we have $e^{x-1}\geq(1+x)/2$ (for example, by comparing Taylor series).

If $m/z\geq1$ (so $m\geq z\geq1$) Proposition~\ref{prop:bm} yields: 
\begin{align*}
\frac{\Vol(Q_v^+)}{\Vol(Q)}\geq&\frac{1/z}{(1/z+1)e^{1+(m/z-1)/(1/z)}}=\frac{1}{(z+1)e^{1+m-z}}\\
=&\frac{e^{z-1}}{z+1}\frac{1}{e^m}\geq\frac{1}{2e^m}.
\end{align*}
If instead $m/z\leq1$ (so $z\geq m\geq1$) Proposition~\ref{prop:bm} still yields:
\begin{align*}
\frac{\Vol(Q_v^+)}{\Vol(Q)}\geq&\frac{1/z+1-m/z}{(1/z+1)e}=\frac{z+1-m}{(z+1)e}=\frac{1}{e}\left(1-\frac{m}{z+1}\right)\\
\geq&\frac{1}{e}\left(1-\frac{m}{m+1}\right)=\frac{e^{m-1}}{m+1}\frac{1}{e^m}\geq\frac{1}{2e^m}. \qedhere
\end{align*}
\end{proof}

\subsection{A uniform geometric approach via generalized order polytopes}\label{sec:order-polytopes-bound}
In order to apply Proposition~\ref{prop:bm}, we need some convex bodies. In this section, we define an analog of the order polytope of a poset (see \cite{Stanley-poset-polytopes}) for convex sets in any finite Weyl group, and apply these to finish the proof of Theorem~\ref{thm:uniform-bound}.

Let $\Phi$ be a root system of rank $r$ with highest root $\xi$. The \emph{fundamental alcove} is
\[
Q_{\id}\coloneqq\{x\in E\:|\:\langle x,\alpha\rangle\geq0\text{ for }\alpha\in\Phi^+,\langle x,\xi\rangle\leq1\}
\]
which lives inside the fundamental Weyl chamber $\overline{R_{\id}}$. For $w\in W(\Phi)$, define $$Q_{w}\coloneqq w^{-1}Q_{\id}=\{x\in E\:|\:\langle x,w^{-1}\alpha\rangle\geq0\text{ for }\alpha\in\Phi^+,\langle x,w^{-1}\xi\rangle\leq1\}\subseteq\overline{R_w}.$$

\begin{def-prop}\label{def:order-polytope}
Let $C\subseteq W(\Phi)$ be a convex set. The set
$$\O(C)\coloneqq \bigcup_{w\in C}Q_{w}$$
is a polytope called the \emph{generalized order polytope} of $C$.
\end{def-prop}
\begin{remark}
Our generalized order polytope $\O(C)$ is a special case of the \emph{alcoved polytopes} studied by Lam and Postnikov \cite{Lam-Postnikov}, which arise naturally from affine Coxeter arrangements.
\end{remark}

\begin{proof}
We need to show that $\O(C)$ is indeed a polytope, and we achieve this by listing the defining half-spaces of $\O(C)$. Suppose that
\[
C=W_D^A=\{w\in W\:|\: D\subseteq I_{\Phi}(w)\subseteq A\}.
\]
Then we already know that inequalities $\langle x,\alpha\rangle\leq0$ for $\alpha\in D$ and $\langle x,\beta\rangle\geq0$ for $\beta\in \Phi^+\setminus A$ cut out $\bigcup_{w\in C}\overline{R_w}$. Let $\O(C)'$ be the polytope cut out by the additional half-spaces $\langle x,w^{-1}\xi\rangle\leq1$ for all $w\in C$. We will show that $\O(C)'=\O(C)$.

Take any $x\in\O(C)$. We may assume without loss of generality that $x\in Q_{\id}$, meaning that $\langle x,\xi\rangle\leq1$ and $\langle x,\alpha\rangle\geq0$ for all $\alpha\in \Phi^+$. It suffices to show that $\langle x,w^{-1}\xi\rangle\leq1$. If $w^{-1}\xi\in\Phi^-$, then $\langle x,w^{-1}\xi\rangle\leq0<1$. If $w^{-1}\xi\in\Phi^+$, since $\xi$ is the maximum in the root poset $\Phi^+$, $\xi-w^{-1}\xi$ is a nonnegative linear combination of simple roots and thus $\langle x,\xi-w^{-1}\xi\rangle\geq0$, so 
\[
\langle x,w^{-1}\xi\rangle\leq\langle x,\xi\rangle\leq1
\]
as desired. This gives $\O(C)\subseteq\O(C)'$.

Take any $x\in\O(C)'$ and again without loss of generality assume that $x\in R_{\id}$. One of the defining half-spaces of $\O(C)'$ is $\langle x,\xi\rangle\leq1$ so $x\in Q_{\id}\subseteq\O(C)$. This gives $\O(C)'\subseteq\O(C)$.
\end{proof}

Recall that $\omega_1^{\vee},\ldots,\omega_r^{\vee}$ is the dual basis of $\alpha_1,\ldots,\alpha_r$ so we can write
\[
\xi=\sum_{i=1}^r\langle\omega_i^{\vee},\xi\rangle\alpha_i.
\]
We see that $Q_{\id}$ is an $r$-simplex with vertices $\{0\}\cup \{\omega_i^{\vee}/\langle\omega_i^{\vee},\xi\rangle \: | \: i=1,\ldots,r\}$. More explicitly, if we write $\xi=c_1\alpha_1+\cdots+c_r\alpha_r$, then the vertices of $Q_{\id}$ are $0,\omega_1^{\vee}/c_1,\ldots,\omega_r^{\vee}/c_r$. This means that the centroid of $Q_{\id}$ is $\frac{1}{r+1}\sum_{i=1}^r\omega_i^{\vee}/\langle\omega_i^{\vee},\xi\rangle$ and the centroid $o_C$ of the order polytope $\O(C)$ is
\[
o_C = \frac{1}{r+1}\frac{1}{|C|}\sum_{w\in C}w^{-1}\left(\sum_{i=1}^r\frac{\omega_i^{\vee}}{\langle\omega_i^{\vee},\xi\rangle}\right).
\]

\begin{lemma}\label{lem:order-polytope-centroid}
Let $C\subseteq W$ be a non-singleton convex set whose order polytope $\O(C)$ has centroid $o_C$. There exists $\beta\in\Phi^+$ such that $\emptyset\neq C_{\beta}\subsetneq C$ and $|\langle o_C,\beta\rangle|\leq m/(r+1)$ where 
\begin{align*}
m&=\frac{r}{m_0}+\frac{1}{m_1}-\frac{\height(\Phi)}{m_0m_1},\\
m_0&=\min_{i\in[r]}\langle\omega_i^{\vee},\xi\rangle,\\ m_1&=\max_{i\in[r]}\langle\omega_i^{\vee},\xi\rangle.
\end{align*}

\end{lemma}
Before the proof, we notice that $\height(\Phi)=\sum_{i=1}^r\langle\omega_i^{\vee},\xi\rangle$ is a sum of $r$ numbers with minimum $m_0$ and maximum $m_1$. This means $\height(\Phi)\leq m_0+(r-1)m_1$ and thus
$$m=\frac{1}{m_0}+\frac{m_0+(r-1)m_1-\height(\Phi)}{m_0m_1}\geq\frac{1}{m_0}\geq\frac{1}{m_1}.$$

\begin{proof}
The proof resembles that of Lemma~\ref{lem:select-root}. The condition that $\emptyset\neq C_{\beta}\subsetneq C$ says precisely that there are chambers lying on both sides of the hyperplane $H_{\beta}$ in $\O(C)$, so let us relax the requirement $\beta\in\Phi^+$ to $\beta\in\Phi$. 

If $\langle o_C,\beta\rangle=0$ for some $\beta\in\Phi$ then we are done, since there must be alcoves on both sides of the hyperplane $H_{\beta}$ for the centroid to lie on this hyperplane, so assume that $\langle o_C,\beta\rangle\neq0$ for all $\beta\in\Phi$. This implies that $\Phi^+_{o_C}\coloneqq\{\beta\in\Phi\:|\: \langle o_C,\beta\rangle>0\}$ is a system of positive roots; write $\gamma_1,\ldots,\gamma_r$ for the simple roots. Let $A$ be the set of $\gamma_j$'s such that there are alcoves on both sides of the hyperplane $H_{\gamma_j}$ in $\O(C)$. 

Suppose that $\gamma_i \not \in A$, meaning that $w\gamma_i\in\Phi^+$ for all $w\in C$ or $w\gamma_i\in\Phi^-$ for all $w\in C$. Then, noting that
$$(r+1)\langle o_C,\gamma_i\rangle=\frac{1}{|C|}\sum_{w\in C}\left\langle w^{-1}\sum_{i=1}^r\frac{\omega_i^{\vee}}{\langle\omega_i^{\vee},\xi\rangle},\gamma_i\right\rangle=\frac{1}{|C|}\sum_{w\in C}\left\langle \sum_{i=1}^r\frac{\omega_i^{\vee}}{\langle\omega_i^{\vee},\xi\rangle},w\gamma_i\right\rangle,$$
and that by assumption each term in the sum has the same sign, we conclude that in fact $w\gamma_i\in\Phi^+$ for all $w\in C$, since $\langle o_C,\gamma_i\rangle>0$. Since $|C| \geq 2$, not all $w \in C$ send $\Phi_{o_C}^{+}$ to $\Phi^+$. Thus $A$ is nonempty. Continuing with the above computation, we see that
\begin{equation}
\label{eq:m1-bound}
(r+1)\langle o_C,\gamma_i\rangle\geq 1/m_1
\end{equation}
since $\langle\sum_{i=1}^r\omega_i^{\vee}/\langle\omega_i^{\vee}/\xi\rangle,\beta\rangle\geq 1/m_1$ for all $\beta\in\Phi^+$.

Assume that $\langle o_C,\gamma_j\rangle>m/(r+1)$ for all $\gamma_j\in A$ and let $\xi_{\gamma}$ be the highest root with respect to the positive roots $\Phi_{o_C}^+$. Then
\begin{align*}
(r+1)\langle o_C,\xi_{\gamma}\rangle &=(r+1)\sum_{i=1}^r \langle \omega_i^{\vee},\xi\rangle\langle o_C,\gamma_i\rangle\\ 
&=(r+1)\sum_{\gamma_i \in A} \langle \omega_i^{\vee},\xi\rangle\langle o_C,\gamma_i\rangle + (r+1)\sum_{\gamma_i \not \in A} \langle \omega_i^{\vee},\xi\rangle\langle o_C,\gamma_i\rangle \\
&>m_0\cdot m+(\height(\Phi)-m_0)\cdot\frac{1}{m_1} \\
&=r,
\end{align*}
where in the inequality we have used (\ref{eq:m1-bound}), the fact that $A$ is nonempty, and the fact that $m \geq \frac{1}{m_1}$. But this is impossible since for any $\beta\in\Phi$
$$(r+1)\langle o_C,\beta\rangle=\frac{1}{|C|}\sum_{w\in C}\left\langle \sum_{i=1}^r\frac{\omega_i^{\vee}}{\langle\omega_i^{\vee},\xi\rangle},w\beta\right\rangle\leq \frac{1}{|C|}\sum_{w\in C}\left\langle \sum_{i=1}^r\frac{\omega_i^{\vee}}{\langle\omega_i^{\vee},\xi\rangle},\xi\right\rangle=r. $$
Thus one of the $\gamma_j \in A$ must have $\langle o_C,\gamma_j\rangle \leq m/(r+1)$.
\end{proof}
The expression for $m$ in Lemma~\ref{lem:order-polytope-centroid} does not look nice, but it is very reasonable in practice. In particular, for all of the infinite families of finite Weyl groups the value of $m$ does not depend on the rank $r$. See Table~\ref{tab:centroid-data} for a complete list of the values of $m$.
\begin{table}[ht!]
\centering
\begin{tabular}{|c|c|c|c|c||c|c|}
\hline
Type & $m_0$ & $m_1$ & $\height(\Phi)$ & $m$ & $mm_1$ & $b(C)\geq$\\\hline
$A_r$ & 1 & 1 & $r$ & 1 & 1 & $1/2e$  \\\hline
$B_r$ & 1 & 2 & $2r-1$ & 1 & 2 & $1/2e^2,1/2e$ \\\hline
$C_r$ & 1 & 2 & $2r-1$ & 1 & 2 & $1/2e^2,1/2e$ \\\hline
$D_r$ & 1 & 2 & $2r-3$ & 2 & 4 & $1/2e^4$\\\hline
$E_6$ & 1 & 3 & 11 & 8/3 & 8 & $1/2e^8$\\\hline
$E_7$ & 1 & 4 & 17 & 3 & 12 & $1/2e^{12}$\\\hline
$E_8$ & 2 & 6 & 29 & 7/4 & 21/2 & $1/2e^{10.5}$\\\hline
$F_4$ & 2 & 4 & 11 & 7/8 & 7/2 & $1/2e^{3.5}$ \\\hline
$G_2$ & 2 & 3 & 5 & 1/2 & 5/2 & $1/2e^{2.5}, 1/3$ \\\hline
\end{tabular}
\caption{The parameters and bounds appearing in Theorem~\ref{thm:uniform-bound-specific} for each irreducible type.  The second listed lower bounds for types $B_r$ and $C_r$ follow from Theorem~\ref{thm:typeB-better-bound}, while the bound of $1/3$ for type $G_2$ is immediate for any rank-two Coxeter group.}
\label{tab:centroid-data}
\end{table}

\begin{theorem}\label{thm:uniform-bound-specific}
Let $C\subseteq W(\Phi)$ be a non-singleton convex set. Then, using the notation of Lemma~\ref{lem:order-polytope-centroid}, \[
b(C)\geq1/2e^{mm_1}.
\]
\end{theorem}
\begin{proof}
We first observe that for any alcove $Q_w$ and any root $\alpha\in\Phi$, there exists $x\in Q_w$ such that $|\langle x,\alpha\rangle|\geq 1/m_1$. To see this, assume $w=\id$ and $\alpha\in\Phi^+$ without loss of generality. Recall that $Q_{\id}$ has vertices at 0 and $\omega_i^{\vee}/\langle\omega_i^{\vee},\xi\rangle$. Let $x=\omega_i^{\vee}/\langle\omega_i^{\vee},\xi\rangle$ for some $i$ such that $\alpha\geq\alpha_i$. Then 
\[
\langle x,\alpha\rangle\geq 1/\langle\omega_i^{\vee},\xi\rangle\geq 1/m_1.
\]

Now, using Lemma~\ref{lem:order-polytope-centroid}, pick $\beta\in\Phi^+$ such that $\emptyset\neq C_{\beta}\subsetneq C$ and $|\langle o_C,\beta\rangle|\leq m/(r+1)$. Since $\emptyset\neq C_{\beta}$, there exists $w\in C_{\beta}$, meaning that $\beta\in I_{\Phi}(w)$ and every $t\in Q_w$ satisfies $\langle t,\beta\rangle\leq0$. By the previous paragraph, there exists $x\in\O(C)$ such that $\langle x,\beta\rangle\leq -1/m_1$. Similarly, as $C_{\beta}\neq C$, there exists $y\in\O(C)$ such that $\langle y,\beta\rangle\geq 1/m_1$. Equivalently, since $\O(C)$ is convex, there are $x,y\in\O(C)$ such that $\langle x,\beta\rangle=-1/m_1$ and $\langle y,\beta\rangle=1/m_1$. If we scale $\O(C)$ by a factor of $m_1$ we obtain a polytope $\O(C)'$ with centroid $o_{C}'$ satisfying $|\langle o_{C}',\beta\rangle|\leq mm_1/(r+1)$ with $mm_1\geq1$ and such that $\min_{x\in \O(C)'}\langle x,\beta\rangle\leq-1$ and $\max_{y\in\O(C)'}\langle y,\beta\rangle\geq 1$.  Then by Corollary~\ref{cor:bm}:
\[
\frac{\Vol(\O(C)_{\beta}^+)}{\Vol(\O(C))}=\frac{\Vol({\O(C)'}_{\beta}^+)}{\Vol(\O(C)')}\geq\frac{1}{2e^{mm_1}}.
\]
By reflecting the polytope, we have 
\[
\frac{\Vol(\O(C)_{\beta}^-)}{\Vol(\O(C))}\geq\frac{1}{2e^{mm_1}}
\]
as well. At the same time, 
\begin{align*}
\Vol(\O(C)_{\beta}^-)&=\Vol(Q_{\id})\cdot |C_{\beta}|, \\ \Vol(\O(C)_{\beta}^+)&=\Vol(Q_{\id})\cdot |C\setminus C_{\beta}|.
\end{align*}
As a result, $\delta_C(\beta)$ witnesses the fact that $b(C)\geq 1/2e^{mm_1}$.
\end{proof}

\subsection{A special treatment for type $B_n$: a short-root order polytope}\label{sec:short-root-order-polytope}
In this section, we are able to improve the bound for type $B_n$ (and thus for type $C_n$ as well) due to the existence of another polytope similar to the order polytope. 

Let $\Phi$ be an irreducible root system of rank $r$ that is not simply laced (that is, $\Phi$ is of type $B_n,C_n,F_4,$ or $G_2$). Then there exists a \emph{highest short root} $\eta\in\Phi^+$ (see for example \cite[\S 3.20]{Humphreys}). In other words, for every root $\alpha\in\Phi^+$ of the same (Euclidean) length as $\eta$, we have $\eta\geq\alpha$ in the root poset. In particular, since the inner product $\langle \cdot , \cdot \rangle$ on $E$ is $W$-invariant, $\eta-w\eta$ is a nonnegative (integral) linear combination of the simple roots for any $w \in W$. As in Section~\ref{sec:order-polytopes-bound}, define
$$Q^{\eta}_{\id}\coloneqq\{x\in E\:|\:\langle x,\alpha \rangle \geq0\text{ for }\alpha\in\Phi^+,\langle x,\eta\rangle\leq1\}$$
and $Q_w^{\eta}\coloneqq w^{-1}Q_{\id}^{\eta}$. 

\begin{def-prop}
Let $C\subseteq W(\Phi)$ be a convex subset of a finite Weyl group that is not simply laced. The set
$$\O^{\eta}(C)\coloneqq\bigcup_{w\in C}Q^{\eta}_w,$$
is a polytope called the \emph{short-root order polytope} of $C$.
\end{def-prop}
The proof that $\O^{\eta}(C)$ is a polytope is exactly the same as that of Definition-Proposition~\ref{def:order-polytope}, using the fact that $\eta-w\eta$ is a nonnegative linear combination of the simple roots, so we omit it.

\begin{theorem}\label{thm:typeB-better-bound}
Let $C\subseteq W(B_n)$ be a non-singleton convex subset of the Weyl group of type $B_n$. Then 
\[
b(C)\geq1/2e.
\]
\end{theorem}
\begin{proof}
In type $B_n$, $\eta=\alpha_1+\alpha_2+\cdots+\alpha_n$ and vertices of $Q^{\eta}_{\id}$ are precisely the origin $0$ and the fundamental coweights $\omega_1^{\vee},\omega_2^{\vee},\ldots,\omega_n^{\vee}$. The centroid of $Q^{\eta}_{\id}$ is thus $(\omega_1^{\vee}+\cdots+\omega_n^{\vee})/(n+1)$ and so the centroid $o_C$ of $\O^{\eta}(C)$ is
\[
o_C=\frac{1}{n+1}\frac{1}{|C|}\sum_{w\in C}w^{-1}(\omega_1^{\vee}+\cdots+\omega_n^{\vee}).
\]
For $\beta\in\Phi^+$, 
$$(n+1)\langle \beta,o_C\rangle=\frac{1}{|C|}\sum_{w\in C}\langle\beta, w^{-1}(\omega_1^{\vee}+\cdots+\omega_n^{\vee})\rangle=\frac{1}{|C|}\sum_{w\in C}\langle w\beta,\omega_1^{\vee}+\cdots+\omega_n^{\vee}\rangle$$
which is precisely $h(\beta)$ as defined in Lemma~\ref{lem:select-root}. By Lemma~\ref{lem:select-root}, we can choose $\beta\in\Phi^+$ such that $|h(\beta)|<1$. If $C_{\beta}=\emptyset$, then $w\beta$ is a positive root for every $w\in C$. This means $\langle w\beta,\omega_1^{\vee}+\cdots+\omega_n^{\vee}\rangle\geq1$ for every $w\in C$, yielding $h(\beta)\geq1$, a contradiction. Thus $C_{\beta}$ is nonempty. For an element $w \in C_{\beta}$, $w\beta$ is a negative root so $\langle\beta,w^{-1}\omega_i^{\vee}\rangle=\langle w\beta,\omega_i^{\vee}\rangle\in\{-1,-2\}$ for some $i$. As $w^{-1}\omega_i^{\vee}$ is a vertex of $Q^{\eta}_{w}\subseteq \O^{\eta}(C)$, there is a point $x\in \O^{\eta}(C)$ such that $\langle x,\beta\rangle\leq -1$. Similarly, if $C_{\beta}=C$ then $h(\beta)\leq-1$ which is impossible. Therefore there exists some $w\in C\setminus C_{\beta}$, which provides us with a vertex in $y\in Q^{\eta}_{w}\subseteq\O^{\eta}(C)$ with $\langle y,\beta\rangle\geq1$. By Corollary~\ref{cor:bm}, as $|\langle o_C,\beta\rangle|=\frac{|h(\beta)|}{n+1}<\frac{1}{n+1}$,
$$\frac{|C\setminus C_{\beta}|}{|C|}=\frac{\Vol(\O^{\eta}(C)_{\beta}^{+})}{\Vol(\O^{\eta}(C))},\frac{|C_{\beta}|}{|C|}=\frac{\Vol(\O^{\eta}(C)_{\beta}^{-})}{\Vol(\O^{\eta}(C))}\geq\frac{1}{2e}.\qedhere $$
\end{proof}
\begin{remark}
Theorem~\ref{thm:typeB-better-bound} improves the bound for $b(C)$ in type $B_n$ to match the bound for type $A_n$ obtained by Kahn and Linial \cite{Kahn-Linial}; we automatically obtain the same bound for type $C_n$ because the Weyl groups of types $B_n$ and $C_n$ are isomorphic as Coxeter groups.  It is possible to perform an analysis similar to that in Theorem~\ref{thm:uniform-bound-specific} with the short-root order polytope in any non-simply-laced type. However no improvement is obtained for types $C_n$ or $F_4$, and it is easy to show directly that the balancing constant is at least $1/3$ for type $G_2$. 
\end{remark}

\section{Some instructive examples}
\label{sec:examples}

In this section we give several examples which show that the natural analogs of Conjecture~\ref{conj:1-3-2-3-weyl} and Theorem~\ref{thm:fully-comm} do not hold in the greatest possible generality of arbitrary convex sets $C$ in arbitrary Coxeter groups.

\begin{ex}
\label{ex:no-general-coxeter-bound}
For $n=3,4,\ldots$, let $W_n$ denote the Coxeter group whose Coxeter diagram is a complete graph $K_n$ (and any edge labels $m_{ij} \geq 3$).  Let $s_1,\ldots, s_n$ be the simple reflections, and let $C_n=\conv(s_1,\ldots,s_n) \subseteq W_n$.

It is easy to see that $C_n=\{\id, s_1, \ldots, s_n\}$ so that $\delta_{C_n}(s_i)=\frac{1}{n+1}$ for all $i$, while $\delta_{C_n}(t)=0$ for $t \not \in \{s_1,\ldots,s_n\}$.  Thus 
\[
b(C_n)=\frac{1}{n+1}.
\]
This example shows that there can be no uniform bound $b(C) \geq \varepsilon > 0$ valid for all non-singleton convex sets in all Coxeter groups.
\end{ex}

\begin{ex}
\label{ex:f-c-fails-in-4-cycle}
Let $W$ be the affine Weyl group $\widetilde{A_3}$, whose Coxeter diagram is a 4-cycle with all edge labels $m_{ij}=3$.  Let $w=s_2s_4s_1s_3$ be the fully commutative element whose heap poset $H_w$ is shown on the right below.
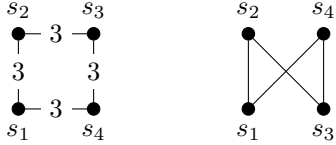
\begin{center}
    \begin{tikzpicture}
\node[draw,shape=circle,fill=black,scale=0.5](a)[label=below: {$s_1$}] at (-1,0) {};
\node[draw,shape=circle,fill=black,scale=0.5](b)[label=above: {$s_2$}] at (-1,1) {};
\node[draw,shape=circle,fill=black,scale=0.5](c)[label=above: {$s_3$}] at (0,1) {};
\node[draw,shape=circle,fill=black,scale=0.5](d)[label=below: {$s_4$}] at (0,0) {};

\draw (a)--(b) node [midway, fill=white] {$3$};
\draw (b)--(c) node [midway, fill=white] {$3$};
\draw (c)--(d) node [midway, fill=white] {$3$};
\draw (d)--(a) node [midway, fill=white] {$3$};
\end{tikzpicture}
\hspace{0.5in}
\begin{tikzpicture}
\node[draw,shape=circle,fill=black,scale=0.5](a)[label=below: {$s_1$}] at (0,0) {};
\node[draw,shape=circle,fill=black,scale=0.5](b)[label=below: {$s_3$}] at (1,0) {};
\node[draw,shape=circle,fill=black,scale=0.5](c)[label=above: {$s_2$}] at (0,1) {};
\node[draw,shape=circle,fill=black,scale=0.5](d)[label=above: {$s_4$}] at (1,1) {};

\draw (c)--(a)--(d);
\draw (c)--(b)--(d);
\end{tikzpicture}
\end{center}
A direct enumeration shows that 
\[
b([\id,w]_L)=\bi(H_w)=\frac{2}{7} < \frac{1}{3}.
\]
Thus Theorem~\ref{thm:fully-comm} may fail when $W$ is not acyclic.
\end{ex}

\begin{ex}
\label{ex:f-c-bound-but-not-3d}
Let $W$ be the rank-four Coxeter group whose Coxeter diagram is below on the left (the $\infty$-edges may be replaced with any labels $m_{12} \geq 4, m_{23} \geq 7, m_{34} \geq 4$).
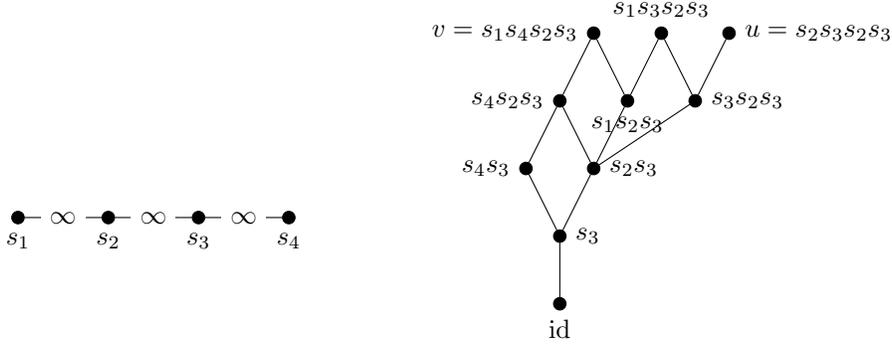
\begin{center}
    \begin{tikzpicture}[scale=.8]
\node (x) at (0,0) {};
\node[draw,shape=circle,fill=black,scale=0.5](a)[label=below: {$s_1$}] at (-1.5,2) {};
\node[draw,shape=circle,fill=black,scale=0.5](b)[label=below: {$s_2$}] at (0,2) {};
\node[draw,shape=circle,fill=black,scale=0.5](c)[label=below: {$s_3$}] at (1.5,2) {};
\node[draw,shape=circle,fill=black,scale=0.5](d)[label=below: {$s_4$}] at (3,2) {};

\draw (a)--(b) node [midway, fill=white] {$\infty$};
\draw (b)--(c) node [midway, fill=white] {$\infty$};
\draw (c)--(d) node [midway, fill=white] {$\infty$};
\end{tikzpicture}
\hspace{0.5in}
\begin{tikzpicture}[scale=.9]
\node[draw,shape=circle,fill=black,scale=0.5](a)[label=below: {$\id$}] at (0,0) {};

\node[draw,shape=circle,fill=black,scale=0.5](b)[label=right: {$s_3$}] at (0,1) {};

\node[draw,shape=circle,fill=black,scale=0.5](c1)[label=left: {$s_4s_3$}] at (-.5,2) {};
\node[draw,shape=circle,fill=black,scale=0.5](c2)[label=right: {$s_2s_3$}] at (.5,2) {};

\node[draw,shape=circle,fill=black,scale=0.5](d1)[label=left: {$s_4s_2s_3$}] at (0,3) {};
\node[draw,shape=circle,fill=black,scale=0.5](d2)[label=below: {$s_1s_2s_3$}] at (1,3) {};
\node[draw,shape=circle,fill=black,scale=0.5](d3)[label=right: {$s_3s_2s_3$}] at (2,3) {};

\node[draw,shape=circle,fill=black,scale=0.5](e1)[label=left: {$v=s_1s_4s_2s_3$}] at (.5,4) {};
\node[draw,shape=circle,fill=black,scale=0.5](e2)[label=above: {$s_1s_3s_2s_3$}] at (1.5,4) {};
\node[draw,shape=circle,fill=black,scale=0.5](e3)[label=right: {$u=s_2s_3s_2s_3$}] at (2.5,4) {};

\draw (a)--(b);
\draw (c1)--(b)--(c2);
\draw (c1)--(d1)--(c2)--(d2);
\draw (c2)--(d3);
\draw (d1)--(e1)--(d2)--(e2)--(d3)--(e3);
\end{tikzpicture}
\end{center}
Let $u=s_2s_3s_2s_3$ and $v=s_1s_4s_2s_3$ and let $C=\conv(\id, u, v)$, shown above right as an order ideal in left weak order.  Then the inversions $t$ with $\delta_C(t)$ closest to $\frac{1}{2}$ are 
\begin{align*}
    \delta_C(s_3s_2s_3s_2s_3)=\delta_C(s_3s_4s_3)=\delta_C(s_3s_2s_1s_2s_3)&=\frac{3}{10}, \\
    \delta_C(s_3s_2s_3)&=\frac{7}{10},
\end{align*}
so $b(C)=\frac{3}{10} < \frac{1}{3}$.  Since $W$ is acyclic, Theorem~\ref{thm:fully-comm} implies that $b([\id,w]_L) \geq \frac{1}{3}$ for any fully commutative element $w$ (if all edge labels are $\infty$, any element $w \in W$ is fully commutative).  Thus this example defies the heuristic (generally believed for the symmetric group \cite{Brightwell}, and conjectured above for finite Weyl groups) that fully commutative intervals realize the smallest balance constants across all convex sets $C \subseteq W$.
\end{ex}

\bibliographystyle{plain}
\bibliography{revision2}
\end{document}